\documentclass[12pt, reqno]{amsart}
\usepackage{amscd,amsmath,amsthm,amssymb,graphics}
\usepackage{amsfonts,amssymb,amscd,amsmath,enumitem,verbatim}
\usepackage[a4paper,top=3cm,left=3cm,right=3cm]{geometry}
\theoremstyle{plain}
\usepackage{color}
\usepackage{hyperref}
\newtheorem{Theorem}{Theorem}
\newtheorem{Lemma}[Theorem]{Lemma}
\newtheorem{Corollary}[Theorem]{Corollary}
\newtheorem{Proposition}[Theorem]{Proposition}

\newtheorem{Conjecture}[Theorem]{Conjecture}

\theoremstyle{definition}
\newtheorem{Definition}[Theorem]{Definition}
\newtheorem{Remark}[Theorem]{Remark}

\newtheorem*{Notation}{Notation}
\newcommand{\Z}{\mathbb{Z}}
\newcommand{\N}{\mathbb{N}}

\usepackage{graphicx}

\newtheorem{innercustomgeneric}{\customgenericname}
\providecommand{\customgenericname}{}
\newcommand{\newcustomtheorem}[2]{%
  \newenvironment{#1}[1]
  {%
   \renewcommand\customgenericname{#2}%
   \renewcommand\theinnercustomgeneric{##1}%
   \innercustomgeneric
  }
  {\endinnercustomgeneric}
}

\newcustomtheorem{customthm}{Theorem}

\title[Transcendence criteria for multidimensional continued fractions]{Transcendence criteria for multidimensional continued fractions}
\author[F. Accossato]{Federico Accossato}
\address{Dipartimento di Scienze Matematiche "Giuseppe Luigi Lagrange", Politecnico di Torino, Corso Duca degli Abruzzi 24, 10129, Torino (TO), Italy}
\email{federico.accossato@polito.it}

\author[N. Murru]{Nadir Murru}
\address{Dipartimento di Matematica, Università di Trento, Via Sommarive 14, 38123, Povo (TN), Italy}
\email{nadir.murru@unitn.it}

\author[G. Romeo]{Giuliano Romeo}
\address{Dipartimento di Scienze Matematiche "Giuseppe Luigi Lagrange", Politecnico di Torino, Corso Duca degli Abruzzi 24, 10129, Torino (TO), Italy}
\email{giuliano.romeo@polito.it}


\begin{document}

\begin{abstract}
Classical results on Diophantine approximation, such as Roth's theorem, provide the most effective techniques for proving the transcendence of special kinds of continued fractions. Multidimensional continued fractions are a generalization of classical continued fractions, introduced by Jacobi, and there are many well-studied open problems related to them. In this paper, we establish transcendence criteria for multidimensional continued fractions. In particular, we show that some Liouville-type and quasi-periodic multidimensional continued fractions are transcendental. We also obtain an upper bound on the naive height of cubic irrationals arising from periodic multidimensional continued fractions and exploit it to prove the transcendence criteria in the quasi-periodic case.
%As main tools, we prove a bound on the growth of the denominators of the convergents of algebraic numbers, which is the analogue of a famous result by Davenport and Roth, and we obtain an upper bound on the naive height of cubic irrationals arising from periodic multidimensional continued fractions.
\end{abstract}

\maketitle

\section{Introduction}
In 1844, Liouville \cite{LIO} constructed the first example of transcendental numbers by using the properties of continued fractions. Liouville proved that algebraic numbers admit only finitely many ``good" approximations and used continued fractions to construct some real numbers that are approximated ``too well" to be algebraic. These continued fractions are usually referred as \textit{Liouville-type} continued fractions, and they have unbounded partial quotients that increase very fast.
The first transcendental continued fractions with bounded partial quotients have been defined by Maillet \cite{MAI} and Baker \cite{BAK}. These are called \textit{quasi-periodic} continued fractions, as they ``mimic" the behavior of periodic continued fractions. Indeed, these continued fractions have same blocks of partial quotients that repeat infinitely many times. Quasi-periodic continued fractions are transcendental as they can be approximated too well by infinitely many quadratic irrationals. See also \cite{MB} for several improvements of the results of Maillet and Baker. These kinds of continued fractions have also been exploited in the framework of $p$-adic continued fractions to prove the transcendence of $p$-adic numbers \cite{LMS, LV, OO} (see also Section 8 of \cite{R}).

Since 2005, Adamczewski and Bugeaud followed up the work started by Maillet and Baker by proving transcendence criteria for several families of continued fractions, including \textit{stammering} continued fractions \cite{AB2}, \textit{palindromic} continued fractions \cite{AB3}, irrational \textit{automatic} numbers \cite{AB} and other continued fractions arising from well-known sequences \cite{ABD, QUE}. The main tools used in their proofs are Roth's theorem \cite{ROTH} and its generalization due to Schmidt, i.e., the Subspace Theorem \cite{SCHM}. In 2013, Bugeaud \cite{BUG2} proved a transcendence criterion for \textit{automatic} continued fractions, including all previous families. The results also imply that the continued fraction expansions of algebraic irrationals is either periodic, and hence it represents a quadratic irrational, or it can not be too ``simple". A famous conjecture states that the continued fraction of any algebraic irrational of degree greater than two contains arbitrarily large partial quotients. Another related conjecture, due to Badziahin and Shallit \cite{BS}, claims that every continued fraction where the terms form a $k$-regular sequence of positive integers is transcendental or quadratic.

Lagrange's theorem for classical continued fractions states that a real number has an eventually periodic continued fraction if and only if it is a quadratic irrational. Multidimensional continued fractions (MCF) have been introduced for the first time by Jacobi \cite{JAC} in order to answer a question posed by Hermite \cite{HER}. The aim was to obtain an analogue of Lagrange's theorem for cubic irrationals. The work of Jacobi has been lately generalized by Perron \cite{PER} to deal with algebraic irrationalities of any degree.
Hermite's problem of finding a periodic continued fraction for algebraic irrationals of degree greater than two is still open (see \cite{OK1, OK2} for some recent progresses).

In this paper, we start to develop the analysis of transcendental continued fractions of higher dimension. We prove several transcendence criteria for general multidimensional continued fractions, mainly inspired by the results of Liouville and the works of Maillet and Baker. More precise results are obtained in the bidimensional case of Jacobi's algorithm. To our knowledge, the only works related to this topic so far are due to Tamura \cite{TAM1,TAM2}, where only very specific MCF expansions were considered. The methodology developed in this paper aims to provide a better understanding of the multidimensional continued fractions of algebraic and transcendental numbers. Some of the results are interesting on their own and can be read independently of the rest of the paper. In particular, Theorem \ref{Thm: height} provides a bound on the height of cubic irrationals corresponding to periodic MCFs.
%and Theorem \ref{Thm: davrothmulti} is an analogue of a famous and powerful result of Davenport and Roth \cite[Theorem 3]{DR}. 

In Section \ref{Sec: multidim}, we review the definition of Jacobi--Perron's algorithm and the basic known properties. 
In Section \ref{Sec: aux}, we prove some new results for MCFs that are exploited for proving the results of the next sections.
In Section \ref{Sec: bounds}, we establish a bound on the height of cubic irrationals having a periodic Jacobi's MCF. In Section \ref{Sec: liouville} and Section \ref{Sec: quasiper}, we prove the transcendence of, respectively, \textit{Liouville-type} and \textit{quasi-periodic} MCF. We also provide an estimate for the growth of the denominators of convergents for an MCF. Section \ref{Sec: openprob} is devoted to some open questions.

%\begin{customthm}{A} \label{thm:A}
%aaa
%\end{customthm}

\section{Multidimensional continued fractions}\label{Sec: multidim}
In this section we recall some basic facts about multidimensional continued fractions and fix the notation. For the general theory, we refer the interested readers to \cite{BER, MUL}. We first introduce Jacobi's algorithm and then we recall its generalization due to Perron.
\subsection{Jacobi's algorithm}\label{Sec: jacobi}
Given $\alpha_0, \beta_0 \in \mathbb R$, Jacobi's algorithm provides two sequences of integers by means of the following recurrences:
\begin{equation} \label{alg-jac} \begin{cases}  a_n = \lfloor \alpha_n \rfloor \cr b_n = \lfloor \beta_n \rfloor \cr \alpha_{n+1} = \cfrac{1}{\beta_n - b_n} \cr \beta_{n+1} = \cfrac{\alpha_n - a_n}{\beta_n - b_n} \end{cases}\end{equation}
for $n \geq 0$. 

The $a_n$'s and $b_n$'s are called \emph{partial quotients}, the $\alpha_n$'s and $\beta_n$'s are called \emph{complete quotients}. The two sequences $(a_n)_{n \geq 0}$ and $(b_n)_{n\geq 0}$ are called a \emph{multidimensional continued fraction} (MCF) and we write 
\[ (\alpha_0, \beta_0) = [(a_0, a_1, \ldots), (b_0, b_1, \ldots)] \]
for indicating that the MCF is obtained starting from $\alpha_0$ and $\beta_0$.

\begin{Remark}\label{Rem:interruptions}
    Note that the algorithm can be performed as long as the $\beta_n$'s are not integers. This is generally referred to as the \textit{simple case}, where two infinite sequences of integers are provided.
    
    If this is not the case, let $n_0$ be the first index such that $\beta_{n_0} \in \mathbb{Z}$. Then, the $a_i$'s, with $i > n_0$, are generated by evaluating the classical continued fraction of $\alpha_{n_0}$ and we say that $ (\alpha_0, \beta_0) = [(a_0, a_1, \ldots, a_{n_0}, \ldots), (b_0, b_1, \ldots, b_{n_0})]$ is the MCF for $(\alpha_0, \beta_0)$, obtained through Jacobi's algorithm.
    If the sequence of the $a_i$'s is infinite, we say that the procedure presents one interruption, while we say that two interruptions occur if both sequences of partial quotients are finite.

    Perron \cite{PER} showed that, if the algorithm admits $l$ interruptions, then there exist $l$ independent relations over $\mathbb{Z}$ occurring between $1, \alpha_0, \beta_0$. The converse holds as well \cite{DFL04}.    
\end{Remark}

From \eqref{alg-jac}, a necessary condition for two sequences of integers $(a_i)_{i\geq 0}$, $(b_i)_{i\geq 0}$ to be obtained by applying Jacobi's algorithm to a certain couple of real numbers $(\alpha_0,\beta_0)$ is that, for all $n \geq 1$,
\begin{equation} \label{Eq: conditionsJP}  
a_n \geq 1, \quad b_n \geq 0, \quad a_n \geq b_n,
\end{equation}
and $b_{n+1} \geq 1$ in the case that $a_n = b_n$. In \cite{PER}, Perron observed that such conditions are also sufficient. For this reason, they are called \textit{admissibility conditions} for the sequences of partial quotients.

Given to a MCF as before, we define the sequences $(A_n)_{n \geq -3}, (B_n)_{n \geq -3}, (C_n)_{n \geq -3}$ as follows:
\[ \begin{cases} A_{-3} = 0, \quad A_{-2} = 0, \quad A_{-1} = 1 \cr B_{-3} = 0, \quad B_{-2} = 1, \quad B_{-1} = 0 \cr C_{-3} = 1, \quad C_{-2} = 0, \quad C_{-1} = 0,\end{cases} \quad \begin{cases} A_n = a_n A_{n-1} + b_n A_{n-2} + A_{n-3} \cr B_n = a_n B_{n-1} + b_n B_{n-2} + B_{n-3} \cr C_n = a_n C_{n-1} + b_n C_{n-2} + C_{n-3}, \end{cases} \]
for all $n \geq 0$. The sequences $(A_n/C_n)_{n \geq 0}, (B_n/C_n)_{n \geq 0}$ are called \emph{convergents} of the MCF. They can be also obtained by
\begin{equation}\label{Eq: matrixform}
\begin{pmatrix} a_0 & 1 & 0 \cr b_0 & 0 & 1 \cr 1 & 0 & 0 \end{pmatrix} \begin{pmatrix} a_1 & 1 & 0 \cr b_1 & 0 & 1 \cr 1 & 0 & 0 \end{pmatrix} \cdots \begin{pmatrix} a_n & 1 & 0 \cr b_n & 0 & 1 \cr 1 & 0 & 0 \end{pmatrix} = \begin{pmatrix} A_n & A_{n-1} & A_{n-2} \cr B_n & B_{n-1} & B_{n-2} \cr C_n & C_{n-1} & C_{n-2} \end{pmatrix},
\end{equation}
for all $n \geq 0$.

We recall now some useful properties of multidimensional continued fractions. 
%In particular, we will show that two multidimensional continued fractions, with the first partial quotients that are equal, converge to two different pairs of real numbers whose distance is small and expressed in terms of denominators of convergents. 

 In the following lemma, we see that the sequences of convergents of a MCF obtained starting from the pair of real numbers $\alpha_0$ and $\beta_0$ actually converge to them. The proof can be found, e.g, in \cite{BER}, but we report it for completeness, since our notation is slightly different and it will be used later.

\begin{Lemma}[\cite{BER}, Theorem 1]\label{Lem: convJ}
Let $(\alpha_0, \beta_0) = [(a_0, a_1, \ldots), (b_0, b_1, \ldots)]$ be an infinite MCF such that the partial quotients satisfy \eqref{Eq: conditionsJP}, then 
\[ \lim_{n \rightarrow \infty} \cfrac{A_n}{C_n} = \alpha_0, \quad \lim_{n \rightarrow \infty} \cfrac{B_n}{C_n} = \beta_0. \]
\end{Lemma}
\begin{proof}
We see the proof only for the convergents $(A_n/C_n)_{n \geq -2}$.
Define
\[ m_k := \min\left( \cfrac{A_k}{C_k}, \cfrac{A_{k+1}}{C_{k+1}}, \cfrac{A_{k+2}}{C_{k+2}} \right), \quad M_k := \max\left( \cfrac{A_k}{C_k}, \cfrac{A_{k+1}}{C_{k+1}}, \cfrac{A_{k+2}}{C_{k+2}} \right),\]
for all $k \geq 0$. We have
\[ \cfrac{A_{k}}{C_{k}} = \cfrac{a_k A_{k-1} + b_k A_{k-2} + A_{k-3}}{C_k} = s_k \cfrac{A_{k-1}}{C_{k-1}} + t_k  \cfrac{A_{k-2}}{C_{k-2}} + w_k  \cfrac{A_{k-3}}{C_{k-3}}, \]
where
\[ s_k := \cfrac{a_k C_{k-1}}{C_k}, \quad t_k := \cfrac{b_k C_{k-2}}{C_k}, \quad w_k := \cfrac{C_{k-3}}{C_k}, \]
for all $k \geq 0$. Observe that
\[s_k + t_k + u_k =1, \quad \cfrac{t_k}{s_k} \leq 1, \quad \cfrac{u_k}{s_k} \leq 1, \quad s_k \geq 1/3,\]
for all $k \geq 0$. Now, we have
\[ \cfrac{A_{k+3}}{C_{k+3}}\geq s_{k+3} m_k + t_{k+3} m_k + u_{k+3} m_k = m_k, \]
and similarly
\[ \cfrac{A_{k+3}}{C_{k+3}} \leq M_k, \]
for all $k \geq 0$. Moreover, since $m_{k+1} \geq m_k$ and $M_{k+1} \leq M_k$, we obtain
\[ 0 \leq m_k \leq m_{k+1} \leq M_{k+1} \leq M_k, \]
for all $k \geq 0$. Thus, the sequence $(m_k)_{k\geq0}$ is a non-decreasing sequence bounded by $M_0$ and the sequence $(M_k)_{k\geq0}$ is a non-increasing sequence bounded by $m_0$ and
\[m:=\lim_{k \rightarrow \infty} m_k, \quad M:=\lim_{k \rightarrow \infty} M_k.\]
Hence, for all $\epsilon > 0$ there exists $k_0 \in \mathbb N$ such that for all $k > k_0$, we have $m_k > m$, i.e.,
\[ \cfrac{A_k}{C_k}, \cfrac{A_{k+1}}{C_{k+1}}, \cfrac{A_{k+2}}{C_{k+2}} > m - \epsilon \]
and
\begin{align*}
\cfrac{A_{k+3}}{C_{k+3}} &> (t_{k+3} + u_{k+3})(m-\epsilon) + s_{k+3}\cdot \cfrac{A_{k+2}}{C_{k+2}} =\\
&=(1-s_{k+3})(m-\epsilon) +s_{k+3}\cdot\cfrac{A_{k+2}}{C_{k+2}} \geq m-\epsilon + s_{k+3}\left(  \cfrac{A_{k+2}}{C_{k+2}} - m\right).
\end{align*}
Hence,
\[\cfrac{A_{k+3}}{C_{k+3}} - m > s_{k+3}\left( \cfrac{A_{k+2}}{C_{k+2}} - m \right) - \epsilon > \cfrac{1}{4}\left( \cfrac{A_{k+2}}{C_{k+2}} - m \right) - \epsilon,\]
from which
\[\cfrac{A_{k+3+h}}{C_{k+3+h}} - m > \cfrac{1}{4^{h+1}}\left( \cfrac{A_{k+2}}{C_{k+2}} - m \right) - \epsilon \left(1+ \cfrac{1}{4}+ \ldots + \cfrac{1}{4^h}\right),\]
for all $h \geq 0$. Now, given $k \geq k_0$ and fixed $h \in \mathbb N$, let $k_1, k_2$ be such that $m_{k_1} = \cfrac{A_{k+3+h}}{C_{k+3+h}}$ and $M_{k_2} = \cfrac{A_{k+2}}{C_{k+2}}$. We have
\begin{eqnarray*} 0 &= m - m > m_{k_1} - m = \cfrac{A_{k+3+h}}{C_{k+3+h}} - m > \cfrac{1}{4^{h+1}}(M_{k_2}-m) - \epsilon \left(1+ \cfrac{1}{4}+ \ldots + \cfrac{1}{4^h}\right) \cr &\geq \cfrac{1}{4^{h+1}}(M-m) - \epsilon \left(1+ \cfrac{1}{4}+ \ldots + \cfrac{1}{4^h}\right). \end{eqnarray*}
From which
\[ \epsilon > \cfrac{M-m}{4^{h+1}(1 + 1/4 + \ldots 1/4^h)}, \]
and since $\epsilon$ is arbitrarily small we must have $m = M$.
\end{proof}

By Lemma \ref{Lem: convJ}, we are therefore allowed to write the following equalities.

\begin{align}\label{Eq: alpha0}
\alpha_0&=a_0+\cfrac{\beta_1}{\alpha_1}=a_0+\cfrac{b_1+\cfrac{1}{\alpha_2}}{a_1+\cfrac{\beta_2}{\alpha_2}}=a_0+\cfrac{b_1+\cfrac{1}{a_2+\ddots}}{a_1+\cfrac{b_2+\ddots}{a_2+\ddots}},\\
\beta_0&=b_0+\cfrac{1}{\alpha_1}=b_0+\cfrac{1}{a_1+\cfrac{\beta_2}{\alpha_2}}=b_0+\cfrac{1}{a_1+\cfrac{b_2+\ddots}{a_2+\ddots}}.\label{Eq: beta0}
\end{align}

\begin{comment}
\begin{Remark}
It is crucial in our analysis to understand the quality of the approximations provided by the convergents $\left(\frac{A_n}{C_n},\frac{B_n}{C_n}\right)$ to the pair $(\alpha,\beta)$. In the classical unidimensional case we know that, for all $n\in\mathbb{N}$,
\begin{equation}\label{Eq: alphaconv}
\left|\alpha-\frac{A_n}{C_n}\right|<\left|\frac{A_{n+1}}{C_{n+1}}-\frac{A_n}{C_n}\right|,
\end{equation}
since $\alpha$ lies always between two consecutive convergents. In the proof of Lemma \ref{Lem: convJ}, we have seen that this is not always the case for multidimensional continued fractions. However, $\alpha$ (and similarly $\beta$) lies between the minimum and the maximum of $3$ consecutive convergents. Therefore, inequality \eqref{Eq: alphaconv} is satisfied at least one out of two steps. In particular, \eqref{Eq: alphaconv} is satisfied for infinitely many $n\in\mathbb{N}$.
\end{Remark}
\end{comment}
To a MCF with partial quotients $(a_n)_{n\geq 0}$ and $(b_n)_{n\geq 0}$, we attach the auxiliary sequences $(\tilde{A}_n)_{n \geq -2}$, $(\tilde{B_n})_{n \geq -2}$, $(\tilde{U_n})_{n \geq -2}$, $(\tilde{\tilde{A}}_n)_{n \geq -1}$, $(\tilde{\tilde{B_n}})_{n \geq -1}$, $(\tilde{\tilde{U_n}})_{n \geq -1}$ defined as
\begin{align*}
\tilde{A}_n&=A_nC_{n-1}-C_nA_{n-1}, \quad \tilde{B}_n=B_nC_{n-1}-B_{n-1}C_{n}, \quad \tilde{U}_n = A_n B_{n-1} - A_{n-1}B_n, \\
\tilde{\tilde{A}}_n&=A_nC_{n-2}-C_nA_{n-2}, \quad \tilde{\tilde{B}}_n=B_nC_{n-2}-B_{n-2}C_{n}, \quad \tilde{\tilde{U_n}} = A_n B_{n-2} - A_{n-2} B_n.
\end{align*}
where $A_n, B_n$ and $C_n$ are numerators and denominators of the $n$-th convergent introduced above.

\begin{Lemma}\label{Lem: Recursion}
The sequences $(\tilde{A}_n)_{n\geq 0}$ and $(\tilde{B}_n)_{n\geq 0}$ satisfy
\begin{align*}
\tilde{A}_n&=-b_n\tilde{A}_{n-1}-a_{n-1}\tilde{A}_{n-2}+\tilde{A}_{n-3},\\
\tilde{B}_n&=-b_n\tilde{B}_{n-1}-a_{n-1}\tilde{B}_{n-2}+\tilde{B}_{n-3},
\end{align*}
for all $n \geq 1$, with initial conditions
\begin{align*}
\begin{cases}
\tilde{A}_{-2} = 0, \quad \tilde{A}_{-1} = 0, \quad \tilde{A}_{0} = -1, \cr  \tilde{B}_{-2} = 1, \quad \tilde{B}_{-1} = 0, \quad \tilde{B}_{0} = 0.
\end{cases}
\end{align*}
Moreover,
\[ |\tilde A_n| \leq C_n, \quad |\tilde B_n| \leq C_n, \quad |\tilde{\tilde{A}}_n| \leq C_n, \quad |\tilde{\tilde{B}}_n| \leq C_n \]
for all $n \geq 0$.
\begin{proof}
See \cite[Chapter 9]{MUL2}.
\begin{comment}
Let us prove the result for $\tilde{A}_n$, the proof for $\tilde{B}_n$ is similar.
By definition of the sequence $\tilde{A}_n$ it is straightforward to check the initial conditions. Then, for all $n\geq 2$,
\begin{align*}
\tilde{A}_n&=A_nC_{n-1}-C_nA_{n-1}=\\
&=(a_nA_{n-1}+b_nA_{n-2}+A_{n-3})C_{n-1}-(a_nC_{n-1}+b_nC_{n-2}+C_{n-3})A_{n-1}=\\
&=b_n(A_{n-2}C_{n-1}-C_{n-2}A_{n-1})+(A_{n-3}C_{n-1}-C_{n-3}A_{n-1})=\\
&=-b_n\tilde{A}_{n-1}+A_{n-3}(a_{n-1}C_{n-2}+b_{n-1}C_{n-3}+C_{n-4})-\\
&\ \  \ -C_{n-3}(a_{n-1}A_{n-2}+b_{n-1}A_{n-3}+A_{n-4})=\\
&=-b_n\tilde{A}_{n-1}+a_{n-1}(A_{n-3}C_{n-2}-C_{n-3}A_{n-2})+(A_{n-3}C_{n-4}-C_{n-3}A_{n-4})=\\
&=-b_n\tilde{A}_{n-1}+a_{n-1}\tilde{A}_{n-1}+\tilde{A}_{n-3},
\end{align*}
and the thesis follows.
\end{comment}
\end{proof}
\end{Lemma}

\subsection{Jacobi--Perron's algorithm}
Jacobi--Perron's algorithm generalizes Jacobi's  and works on a $m$-tuple of real numbers $(\alpha_0^{(1)},\ldots,\alpha_0^{(m)})$ and it provides the multidimensional continued fraction as the $m$-tuple of integer sequences
\[[(a_n^{(1)})_{n\geq 0},\ldots,(a_n^{(m)})_{n\geq 0}],\]
generated by the iteration of the following equations
\begin{align*}
\begin{cases}
a_n^{(i)}=\lfloor \alpha_n^{(i)} \rfloor, \ \ i=1,\ldots,m\\
\alpha_{n+1}^{(1)}=\frac{1}{\alpha_n^{(m)}-a_n^{(m)}}\\
\alpha_{n+1}^{(i)}=\frac{\alpha_n^{(i-1)}-a_n^{(i-1)}}{\alpha_n^{(m)}-a_n^{(m)}}, \ \ i=2,\ldots,m.
\end{cases}
\end{align*}
for all $n \geq 0$. Sometimes, we refer to these MCFs as $m$-dimensional continued fractions. Notice that, when $m = 2$, the algorithm reduces to Jacobi's algorithm and $\alpha_n^{(1)} = \alpha_n$, $\alpha_n^{(2)} = \beta_n$.

As in Remark \ref{Rem:interruptions}, notice that if $\alpha_{n_0}^{(m)}$ is integer for some index $n_0$, then the algorithm has an interruption, and in this case the procedure continues by taking in input the $(m-1)$-uple $(\alpha_{n_0}^{(1)},\dots,\alpha_{n_0}^{(m-1)})$. Also in this case, when one single input is given, the procedure essentially reduces to the algorithm for classical continued fractions.

 Again, \cite{DFL04} is a good reference for the analysis of the connections between the number of interruptions in the algorithm and the number of independent rational relations between $1,\alpha_0^{(i)},\dots,\alpha_0^{(m)}$. We just recall that, if $1,\alpha_0^{(i)},\dots,\alpha_0^{(m)}$ are rationally independent, then no interruption is present (i.e., the sequences of partial quotients are all infinite). The converse is not true in general.

For Jacobi--Perron's algorithm, we still have the following admissibility conditions on the partial quotients:
\begin{equation} \label{eq:cond-pq}
(a_n^{(1)}, a_{n+1}^{(2)}, \ldots, a_{n+i}^{(i+1)}) \succeq (a_n^{(m-i)}, a_{n+1}^{(m-i+1)}, \ldots, a_{n+1-i}^{(m)}, 1),
\end{equation}
for $0 \leq i \leq m-1$ and for all $n \geq 1$, where $\succeq$ denotes the lexicographic order (see, e.g., \cite{DFL04}), and 
\[a_n^{(j)} \geq 0\]
for $1 \leq j \leq m$ and for all $n \geq 1$.
%When $m = 2$, the above admissibility conditions coincide with the conditions described for Jacobi's algorithm.

Also in this case we introduce the sequence of convergents
\((A_n^{(i)}/C_{n})_{n \geq 0},\)
for all $n\in\mathbb{N}$ and for all $i=1,\ldots,m$, where the numerators $A_n^{(i)}$ are defined as follows:
\begin{align} \label{eq:conv-A}
\begin{cases}
A_{-n}^{(i)}=\delta_{in}, \quad n = 1, \ldots, m+1\\
A_n^{(i)}=\sum\limits_{j=1}^{m}a_n^{(j)}A_{n-j}^{(i)}+A_{n-m-1}^{(i)}, \quad n \geq 0,
\end{cases}
\end{align}
and the denominators $C_n$ are defined as
\begin{align}\label{eq:conv-C}
    \begin{cases}
C_{-m-1}=1,\\
C_{-n} = 0, \quad   \textup{for } n=1,\ldots, m   \\ 
C_n=\sum\limits_{j=1}^{m}a_n^{(j)}C_{n-j}+C_{n-m-1}, \quad  \text{for } n\geq 0.
\end{cases}
\end{align}

Again, it is possible to prove that, for every $ i = 1, \ldots, m$,
\begin{equation}\label{eq:lim-conv}
    \lim_{n \rightarrow \infty} \cfrac{A_n^{(i)}}{C_n} = \alpha_0^{(i)},
\end{equation}
thus allowing us to write
\begin{align*} 
\begin{cases}
\alpha_{n}^{(i-1)}=a_{n}^{(i-1)}+\frac{\alpha_{n+1}^{(i)}}{\alpha_{n+1}^{(1)}},  \ \ i=2,\ldots,m\\
\alpha_{n}^{(m)}=a_{n}^{(m)}+\frac{1}{\alpha_{n+1}^{(1)}},
\end{cases}
\end{align*}
which generalize the expressions in \eqref{Eq: alpha0} and \eqref{Eq: beta0}.
\begin{Remark}
When $m$ sequences $[(a_0^{(1)},a_1^{(1)},\dots),\dots,(a_0^{(m)},a_1^{(m)},\dots)]$ of real numbers, not necessarily integers, are given, one can always construct the convergents by using \eqref{eq:conv-A} and \eqref{eq:conv-C}. If \eqref{eq:lim-conv} is satisfied, we still say that $[(a_0^{(1)},a_1^{(1)},\dots),\dots,(a_0^{(m)},a_1^{(m)},\dots)]$ is a MCF converging to $(\alpha_0^{(1)},\dots,\alpha_0^{(m)})$.

Therefore, it can be easily proved that, if \[(\alpha_0^{(1)},\dots,\alpha_0^{(m)}) = [(a_0^{(1)},a_1^{(1)},\dots),\dots,(a_0^{(m)},a_1^{(m)},\dots)],\]
then, for every $k \geq 1$,
\begin{align}
    & (\alpha_k^{(1)},\dots,\alpha_k^{(m)})=[(a_k^{(1)},a_{k+1}^{(1)},\dots),\dots,(a_k^{(m)},a_{k+1}^{(m)},\dots)] \\
    \text{and } & (\alpha_0^{(1)},\dots,\alpha_0^{(m)}) = [(a_0^{(1)},\dots,a_{k-1}^{(1)},\alpha_k^{(1)}),\dots,(a_0^{(m)},\dots,a_{k-1}^{(m)},\alpha_k^{(m)})].
\end{align}
Moreover, for every $i = 1,\dots,m$, it holds that
\begin{equation}\label{eq:frac-alpha0}
    \alpha_0^{(i)} = \cfrac{\alpha_k^{(1)}A_{k-1}^{(i)}+\cdots +\alpha_k^{(m)}A_{k-m}^{(i)}+A_{k-m-1}^{(i)}}{\alpha_k^{(1)}C_{k-1}+\cdots +\alpha_k^{(m)}C_{k-m}+C_{k-m-1}}.
\end{equation}
%Let us observe that a MCF converging to a $m$--tuple of real numbers can be also obtained without using the Jacobi--Perron algorithm for evaluating the partial quotients. In the following, we will only focus on MCFs obtained from the Jacobi--Perron algorithm so that, in particular, the conditions \eqref{eq:cond-pq} always hold.
\end{Remark}

Finally, suppose that $(\alpha_0^{(1)},\dots,\alpha_0^{(m)}) = [(a_0^{(1)},a_1^{(1)},\dots),\dots,(a_0^{(m)},a_1^{(m)},\dots)]$ is a MCF obtained through Jacobi--Perron's algorithm. Then, for every $k \geq 0$, we have
\begin{align}\label{Eq: idk}
\begin{split}
\begin{pmatrix} a_k^{(1)} & 1 & 0 & \cdots & 0 \cr a_k^{(2)} & 0 & 1 & \cdots & 0 \cr \vdots & \vdots & \vdots & \ddots & \vdots \cr a_k^{(m)} & 0 & 0 & \cdots & 1 \cr 1 & 0 & 0 &\cdots & 0 \end{pmatrix}\begin{pmatrix}
\alpha_{k+1}^{(1)}\cr
\alpha_{k+1}^{(2)}\cr
\vdots \cr
\alpha_{k+1}^{(m)}\cr
1
\end{pmatrix}&=\begin{pmatrix}
a_k^{(1)}\alpha_{k+1}^{(1)}+\alpha_{k+1}^{(2)}\cr
a_k^{(2)}\alpha_{k+1}^{(1)}+\alpha_{k+1}^{(3)}\cr
\vdots \cr
a_k^{(m)}\alpha_{k+1}^{(1)}+1\cr
\alpha_{k+1}^{(1)}
\end{pmatrix}= \\
&= \alpha_{k+1}^{(1)}\begin{pmatrix}
a_k^{(1)}+\frac{\alpha_{k+1}^{(2)}}{\alpha_{k+1}^{(1)}}\cr
a_k^{(2)}+\frac{\alpha_{k+1}^{(3)}} {\alpha_{k+1}^{(1)}}\cr
\vdots \cr
a_k^{(m)}+\frac{1} {\alpha_{k+1}^{(1)}}\cr
1
\end{pmatrix}
= \alpha_{k+1}^{(1)}\begin{pmatrix}
\alpha_{k}^{(1)}\cr
\alpha_{k}^{(2)}\cr
\vdots \cr
\alpha_{k}^{(m)} \cr
1
\end{pmatrix}.    
\end{split}
\end{align}
Equation \eqref{Eq: idk} leads to the following identity, which will be useful in Section \ref{Sec: bounds}:
\begin{equation}\label{Eq: pp}
\begin{pmatrix}
\alpha_0^{(1)}\cr
\vdots \cr
\alpha_0^{(m)}\cr
1
\end{pmatrix}=\frac{1}{\alpha_1^{(1)}\alpha_2^{(1)}\ldots\alpha_{k+1}^{(1)}} M_0 M_1 \cdots M_k 
\begin{pmatrix}
\alpha_{k+1}^{(1)}\cr
\vdots \cr
\alpha_{k+1}^{(m)}\cr
1
\end{pmatrix},
\end{equation}
where, for every $n$,
\begin{equation}\label{eq:M_n^m}
    M_n^{(m)} = \begin{pmatrix} a_n^{(1)} & 1 & \cdots & 0  \cr \vdots & \vdots  & \ddots & \vdots \cr a_n^{(m)} & 0 & \cdots & 1 \cr 1  & 0 &\cdots & 0 \end{pmatrix}.
\end{equation}

As we did in the previous section for Jacobi's MCF, we attach to a Jacobi--Perron's MCF with partial quotients $(a_n^{(i)})_{n\geq 0}$, for $i=1,\ldots,m$ and for all $n\in\mathbb{N}$, the auxiliary sequences
\[\tilde{A}^{(i)}_n=A_n^{(i)}C_{n-1}-A_{n-1}^{(i)}C_{n}.\]
\begin{Remark}\label{Rem: nondep}
For the sequences $\tilde{A}^{(i)}_n$ in the $m$-dimensional case we did not find an elegant recursion as for Jacobi's algorithm in Lemma \ref{Lem: Recursion}. However, we can notice that also in this case, each $\tilde{A}^{(i)}_n$ does not depend on $a_n^{(1)}$, since
\begin{align*}
\tilde{A}^{(i)}_n&=A^{(i)}_nC_{n-1}-A^{(i)}_{n-1}C_{n}=(a_n^{(1)}A^{(i)}_{n-1}+\ldots+a_n^{(m)}A^{(i)}_{n-m}+A^{(i)}_{n-m-1})C_{n-1}-\\
&\ \ \ -A_{n-1}(a_n^{(1)}C_{n-1}+\ldots+a_n^{(m)}C_{n-m}+C_{n-m-1})=\\
&=a_n^{(2)}(A_{n-2}C_{n-1}-A_{n-1}C_{n-2})+\ldots+a_n^{(m)}(A_{n-m}C_n-A_{n}C_{n-m})+\\
&\ \ \ +(A_{n-m-1}C_{n-1}-A_{n-1}C_{n-m-1}).
\end{align*}
\end{Remark}

\begin{Notation}
For the sake of simplicity, when dealing with the case $m=2$, throughout the paper we prefer to write $\alpha_n, \beta_n, a_n, b_n, A_n, B_n$, instead of $\alpha_n^{(1)}, \alpha_n^{(2)}, a_n^{(1)}, a_n^{(2)} $, $ A_n^{(1)} ,A_n^{(2)}$.
\end{Notation}

\section{Auxiliary results}\label{Sec: aux}
In this section, we prove some results that are used in the proofs of the main transcendence criteria.

\begin{Lemma} \label{lemma:approx}
Given $(\alpha_0^{(1)},\dots,\alpha_0^{(m)}) = [(a_0^{(1)},a_1^{(1)},\dots),\dots,(a_0^{(m)},a_1^{(m)},\dots)]$, then, for all $i=1,\ldots,m$, 
\[ \left|\alpha_0^{(i)}-\frac{A^{(i)}_n}{C_n}\right|<\frac{|\tilde{A}^{(i)}_{n+1}|}{C_{n+1}C_{n}},\]
for infinitely many $n \in \mathbb N$.
\end{Lemma}
\begin{proof}
Similarly as in the proof of Lemma \ref{Lem: convJ} for the case $m=2$, it is possible to see that, setting for all $i=1,\ldots,m$,
\begin{align*}
m^{(i)}_n &:= \min\left( \cfrac{A^{(i)}_n}{C_n}, \cfrac{A^{(i)}_{n+1}}{C_{n+1}},\ldots, \cfrac{A^{(i)}_{n+m}}{C_{n+m}} \right),\\
M^{(i)}_n &:= \max\left( \cfrac{A^{(i)}_n}{C_n}, \cfrac{A^{(i)}_{n+1}}{C_{n+1}},\ldots, \cfrac{A^{(i)}_{n+m}}{C_{n+m}} \right),
\end{align*}
we have that
\[m^{(i)}_n\leq\alpha_0^{(i)} \leq M^{(i)}_n,\]
for all $n\in\N$ (see Theorem 1 of \cite{BER} for more details).
Therefore, for all $n\in\mathbb{N}$, $\alpha_0^{(i)}$ is contained either in the interval $\left(\frac{A^{(i)}_{n+j}}{C_{n+j}},\frac{A^{(i)}_{n+j+1}}{C_{n+j+1}}\right)$ or $\left(\frac{A^{(i)}_{n+j+1}}{C_{n+j+1}},\frac{A^{(i)}_{n+j}}{C_{n+j}}\right)$, for some $j=0,\ldots,m-1$.
This means that
\[\left|\alpha_0^{(i)}-\frac{A^{(i)}_n}{C_n}\right|<\left|\frac{A^{(i)}_{n+1}}{C_{n+1}}-\frac{A^{(i)}_{n}}{C_{n}}\right|=\frac{|\tilde{A}^{(i)}_{n+1}|}{C_{n+1}C_{n}},\]
for infinitely many $n$ and for all $i=1,\ldots,m$.
\end{proof}

The following two results provide an estimate of the difference between real numbers sharing a certain number of initial partial quotients in their MCF. They hold only in the case when $m=2$, as they both use the inequalities of Lemma \ref{Lem: Recursion}, that are specific for Jacobi's algorithm.

\begin{Lemma}\label{Lem: pqsame}
Given 
$$(\alpha_0,\beta_0) = [(a_0, a_1, \ldots), (b_0, b_1, \ldots)] \quad \text{and} \quad (\alpha'_0,\beta'_0) = [(a'_0, a'_1, \ldots), (b'_0, b'_1, \ldots)],$$ if $a_i = a'_i$ and $b_i = b'_i$ for all $i = 0, \ldots, n$, then
\begin{equation*}
\lvert\alpha_0-\alpha'_0\rvert<\frac{1}{C_{n-2}},\quad
|\beta_0-\beta'_0|<\frac{1}{C_{n-2}}.
\end{equation*}
\end{Lemma}
\begin{proof}
Let us suppose without loss of generality that $\alpha_0'\leq\alpha_0$. Since $\alpha_0$ and $\alpha_0'$ share up to the $n$-th convergent, we know that, for all $n\in\mathbb{N}$,
\begin{equation} 
m_{n+2}< \alpha_0'\leq\alpha_0<M_{n+2}.
\end{equation}
where $m_{n+2}$ and $M_{n+2}$ are as in the proof of Lemma \ref{Lem: convJ}. The thesis then follows from the latter expression, by applying the inequalities of Lemma \ref{Lem: Recursion}. The proof for $\beta_0,\beta_0'$ is similar. 
\end{proof}

\begin{Lemma}\label{Lem: equalpq}
Let $(\alpha,\beta)=[(a_0,a_1,\ldots),(b_0,b_1,\ldots)]$ and $(\alpha',\beta')=[(a_0',a_1',\ldots),(b_0',b_1',\ldots)]$ such that $a_i=a_i'$ and $b_i=b_i'$ for all $i=0,\ldots,n$. Then
\begin{equation}
|\alpha-\alpha'|<\frac{2}{C_n}, \ \ \ \ |\beta-\beta'|<\frac{2}{C_n}.
\end{equation}
\begin{proof}
The thesis follows by writing
\[|\alpha-\alpha'|<\left|\alpha-\frac{A_n}{C_n}\right|+\left|\alpha'-\frac{A_n}{C_n}\right|<\frac{2}{C_n},\]
and the same holds for $\beta,\beta '$.
\end{proof}
\end{Lemma}

\begin{Remark}
Lemma \ref{Lem: pqsame} and Lemma \ref{Lem: equalpq} provide two different approximations for the difference of elements sharing the first partial quotients of their MCF. In fact, although $\frac{2}{C_n}$ is usually smaller than $\frac{1}{C_{n-2}}$, as $C_n=a_nC_{n-1}+b_nC_{n-2}+C_{n-3}$, it is not hard to see that this is not always the case.
\end{Remark}

\begin{Proposition}\label{Lem: numedenN}
Let $(\alpha_0^{(1)},\ldots,\alpha_0^{(m)})\in\mathbb{R}^m$ such that $0<\alpha_0^{(i)}<1$ for all $i=1,\ldots,m$. Then there exists a constant $K\in\N$ such that $A_n^{(i)}\leq KC_n$ for all $n\geq 0$.
\begin{proof}
Let $(M_k^{(i)})_{k\geq 0}$ be the sequence
\[M_k ^{(i)}= \max\left( \cfrac{A_k^{(i)}}{C_k}, \cfrac{A_{k+1}^{(i)}}{C_{k+1}}, \ldots,\cfrac{A_{k+m}^{(i)}}{C_{k+m}} \right),\]
for all $i=1,\ldots,m$. From the proof of Lemma \ref{Lem: convJ} we know that, for all $i$, $(M_k^{(i)})_{k\geq 0}$ is strictly decreasing and converging to $\alpha_0^{(i)}<1$. Therefore, there exists $k_0$ such that $A_k^{(i)}\leq C_k$ for all $k\geq k_0$, and the claim is proved.
\end{proof}
\end{Proposition}

For the $2$-dimensional case, we can do better, as the numerators of the convergents can be explicitly bounded in terms of the denominator.

\begin{Lemma}\label{Lem: numeden}
Let $(\alpha_0,\beta_0)\in\mathbb{R}^2$ such that $0<\alpha_0,\beta_0<1$, then $A_n,B_n\leq C_n$ for all $n\geq 0$.
\begin{proof}
The claim can be easily proved by induction. For $n=0$, $A_0=a_0=0$ and $B_0=b_0=0$, since $0<\alpha_0,\beta_0<1$ so that $1=C_0>A_0,B_0$.
Then we have that $A_1=b_1$, $B_1=1$ and $C_1=a_1$ so that also in this case $C_1\geq A_1, B_1$. At the second step, $A_2=a_2b_1$, $B_2=a_2$ and $C_2=a_2a_1+1$ and  since $a_2\geq b_2$, the base step holds also for $n=2$. Let us now suppose that the thesis holds for $n-3$, $n-2$ and $n-1$ and let us prove it for $n$. By the recursion formula we have
\begin{align*}
A_{n}&=a_nA_{n-1}+b_nA_{n-2}+A_{n-3},\\
B_{n}&=a_nB_{n-1}+b_nB_{n-2}+B_{n-3},\\
C_{n}&=a_nC_{n-1}+b_nC_{n-2}+C_{n-3}.
\end{align*}
Comparing the right-hand sides of the latter equations, we can see that $C_n\geq A_n,B_n$ by inductive hypothesis, and the claim holds.
\end{proof}
\end{Lemma}

Following the same inductive proof of Lemma \ref{Lem: numeden}, it is possible to prove the following general result.

\begin{Corollary}\label{Cor: NM}
Let $(\alpha_0,\beta_0)\in\mathbb{R}^2$ such that
\[N< \alpha_0<N+1,\ \ M<\beta_0<M+1,\]
for some integers $N,M\geq 0 $. Then, for all $n\geq 0$,
\begin{align*}
NC_n\leq A_n<(N+1)C_n,\\
MC_n\leq B_n<(M+1)C_n.
\end{align*}
\end{Corollary}

Lemma \ref{Lem: Recursion} provides the bound $\tilde{A}_n\leq C_n$, for the case $m=2$. In Theorem \ref{Thm: davrothmulti}, we need to bound $\tilde{A}_{n+1}$ using the denominator of the previous step $C_n$. Therefore, we prove the following lemma for Jacobi's MCF, but a similar result can be obtained for $m\geq 3$.

\begin{Lemma}\label{Lem: tilde}
Let $[(a_0,a_1,\ldots),(b_0,b_1,\ldots)]$ such that $a_{n+1}<C_n$ for all $n\in\mathbb{N}$. Then
\begin{equation}
\tilde{A}_{n+1}<3C_n^2, \ \ \ \ \tilde{B}_{n+1}<3C_n^2,
\end{equation}
for all $n\in\mathbb{N}$.
\begin{proof}
We know that, for all $n\in\mathbb{N}$, $\tilde{A}_{n+1}\leq C_{n+1}$. Therefore, we obtain
\[\tilde{A}_{n+1}\leq C_{n+1}=a_{n+1}C_n+b_{n+1}C_{n-1}+C_{n-2}<a_{n+1}(C_n+C_{n-1}+C_{n-2})<3C_n^2,\]
since $b_{n+1}\leq a_{n+1}<C_n$, and similarly for $\tilde{B}_{n+1}$.
\end{proof}
\end{Lemma}

The next lemma shows that periodic Jacobi's MCFs converge to cubic irrationals in the same number field.

\begin{Lemma}\label{Lem: samecubicfield}
Let 
\begin{equation}
(\alpha_0,\beta_0)=[(a_0,\ldots,a_{k-1},\overline{a_k,\ldots,a_{k+h-1}}),(b_0,\ldots,b_{k-1},\overline{b_k,\ldots,b_{k+h-1}})],
\end{equation}
and
\begin{equation}
(\alpha_0',\beta_0')=[(c_0,\ldots,c_{k-1},\overline{a_k,\ldots,a_{k+h-1}}),(d_0,\ldots,d_{k-1},\overline{b_k,\ldots,b_{k+h-1}})].
\end{equation}
Then $\alpha_0, \alpha_0', \beta_0, \beta_0'$ belong to the same cubic number field.
\begin{proof}
Let $(\tilde \alpha, \tilde \beta) = [(\overline{a_k,\ldots,a_{k+h-1}}),(\overline{b_k,\ldots,b_{k+h-1}})]$. Then $\tilde \alpha$ and $\tilde \beta$ are cubic irrationals belonging to the same cubic field (see \cite[Section 3F]{Bren}). Thus, recalling \eqref{eq:frac-alpha0}, the thesis follows from
\[ \alpha_0 = \cfrac{\tilde \alpha A_{k-1} + \tilde \beta A_{k-2} + A_{k-3}}{\tilde \alpha C_{k-1} + \tilde \beta C_{k-2} + C_{k-3}}, \quad \beta_0 = \cfrac{\tilde \alpha B_{k-1} + \tilde \beta B_{k-2} + B_{k-3}}{\tilde \alpha B_{k-1} + \tilde \beta B_{k-2} + B_{k-3}}, \]
\[ \alpha_0' = \cfrac{\tilde \alpha A_{k-1}' + \tilde \beta A_{k-2}' + A_{k-3}'}{\tilde \alpha C_{k-1}' + \tilde \beta C_{k-2}' + C_{k-3}'}, \quad \beta_0' = \cfrac{\tilde \alpha B_{k-1}' + \tilde \beta B_{k-2}' + B_{k-3}'}{\tilde \alpha B_{k-1}' + \tilde \beta B_{k-2}' + B_{k-3}'}, \]
where $\left(\cfrac{A_n}{C_n}, \cfrac{B_n}{C_n}\right)$ and $\left(\cfrac{A_n'}{C_n'}, \cfrac{B_n'}{C_n'}\right)$, for all $n \geq 0$, are the convergents of $(\alpha_0, \beta_0)$ and $(\alpha_0', \beta_0')$, respectively.
\end{proof}
\end{Lemma}

Now, we provide lower bound for the $n$-th convergent of Jacobi's MCF and, in the case of bounded partial quotients, also an upper bound.

\begin{Lemma}\label{Lem: psi}
Let $[(a_0,a_1,\ldots),(b_0,b_1,\ldots)]$ be the sequences of partial quotients of a MCF. Then $C_n>\psi^{n-2}$, where $\psi$ is the only real root of the polynomial $f(x)=x^3-x^2-1$.\\
Moreover, if the partial quotients $a_n$ are bounded by a constant $M$, i.e. $a_n\leq M$ for all $n\in\mathbb{N}$, then $C_n\leq \eta^n$, where $\eta$ is the positive real root of the polynomial $g(x)=x^3-Mx^2-Mx-1$.
\begin{proof}
We prove both the statements by induction. Notice that $\psi\approx 1.46$ so that for $n=0,1,2$ the thesis is true. For the induction step,
\begin{align*}
C_{n+3}=a_{n+3}C_{n+2}+b_{n+3}C_{n+1}+C_n\geq C_{n+2}+C_n\geq \psi^{n-2}(\psi^2+1)=\psi^{n+1},
\end{align*}
since $\psi^3=\psi^2+1$ as a root of $f$.\smallskip

For the second part of the statement, let us assume that $a_n\leq M$ for all $n\in\mathbb{N}$. In general, the polynomial $g$ can have either one or three real roots. It has three real roots if and only if its discriminant is positive, i.e.
\begin{equation}
\Delta (g)=M^4-18M^2-27>0,
\end{equation}
that happens, in particular,  for all $M\geq 5$. It is possible to see that for $1 \leq M \leq 4$ the base step holds. For $M\geq 5$, by Descartes' rule of signs there is only a positive root $\eta$, as the product of all roots is $1$ and there are exactly two roots of same sign. Moreover, for the same reason, $\eta\geq M\geq 1$ as $M$ is the trace of $g$. Therefore,
\[C_{-1}=0\leq \eta^{-1}, \ \ \ C_{0}=1=\eta^0, \ \ \ C_{1}=a_1\leq M \leq\eta,\]
and the base of the induction is proved. For the inductive step notice that, for all $n\in\mathbb{N}$,
\begin{align*}
C_{n+2}&=a_{n+2}C_{n+1}+b_{n+2}C_{n}+C_{n-1}\leq M\eta^{n+1}+M\eta^{n}+\eta^{n-1}=\\
&=\eta^{n-1}(M\eta^2+M\eta+1)=\eta^{n+2},
\end{align*}
as $\eta$ is a root of $g$, and the thesis is true.
\end{proof}
\end{Lemma}

\section{Bounds on the height of cubic irrationals}\label{Sec: bounds}
In this section we prove an explicit bound on the naive height of cubic irrationals corresponding to a periodic Jacobi's MCF.
\bigskip

Let $\gamma$ be an algebraic number and let $P_\gamma$ be its minimal polynomial, that is, the unique polynomial with integer coefficients vanishing at $\gamma$ and irreducible over $\mathbb{Z}$. Then, let $H(\gamma)$ be the naive height of $\gamma$, i.e., the maximum of the absolute values of the coefficients of $P_\gamma$.

In the following theorem we prove that if $(\alpha_0, \beta_0)$ has a periodic MCF expansion, then $\alpha_0$ and $\beta_0$ satisfy a polynomial with integer coefficients of degree three and we find an explicit bound on their naive height in the case $0 < \alpha_0, \beta_0 < 1$. Then, exploiting Corollary \ref{Cor: NM}, the case of general $\alpha_0,\beta_0$ is considered in Theorem \ref{Thm: heightGen},.

\begin{Theorem}\label{Thm: height}
Given $(\alpha_0, \beta_0) \in \mathbb R^2$ such that
\begin{equation}\label{Eq: pureper}
(\alpha_0,\beta_0)=[(a_0,\ldots,a_{k-1},\overline{a_k,\ldots,a_{k+h-1}}),(b_0,\ldots,b_{k-1},\overline{b_k,\ldots,b_{k+h-1}})],
\end{equation}
then $\alpha_0$ and $\beta_0$ are cubic irrationals. Moreover, if $0 < \alpha_0, \beta_0 < 1$, then
\[H(\alpha_0), H(\beta_0)\leq 3024 C_{h+k-1}^9.\]
\end{Theorem}
\begin{proof}
From \eqref{Eq: pureper}, we have that $\alpha_k=\alpha_{k+h}$ and $\beta_k=\beta_{k+h}$, so that, by \eqref{Eq: idk},
\begin{align}\label{Eq: npp}
\begin{pmatrix}
\alpha_k\\
\beta_k\\
1
\end{pmatrix}=&\frac{1}{\alpha_{1}\ldots\alpha_{k+h}}\begin{pmatrix} a_k & 1 & 0 \cr b_k & 0 & 1 \cr 1 & 0 & 0 \end{pmatrix} \cdots \begin{pmatrix} a_{k+h-1} & 1 & 0 \cr b_{k+h-1} & 0 & 1 \cr 1 & 0 & 0 \end{pmatrix}\begin{pmatrix}
\alpha_{k}\\
\beta_{k}\\
1
\end{pmatrix}.
\end{align}
Moreover, by applying \eqref{Eq: pp} with $m=2$,  we get
\begin{align}\label{Eq: reck}
\begin{pmatrix}
\alpha_k\\
\beta_k\\
1
\end{pmatrix}=
\alpha_1\alpha_2\ldots\alpha_{k}\begin{pmatrix} a_{k-1} & 1 & 0 \cr b_{k-1} & 0 & 1 \cr 1 & 0 & 0 \end{pmatrix}^{-1}\cdots \begin{pmatrix} a_0 & 1 & 0 \cr b_0 & 0 & 1 \cr 1 & 0 & 0 \end{pmatrix}^{-1}  \begin{pmatrix}
\alpha_0\\
\beta_0\\
1
\end{pmatrix}.
\end{align}
Following the notation introduced in \eqref{eq:M_n^m}, put $M_n = M_n^{(2)}$ for every $n$. Hence, by plugging \eqref{Eq: reck} into the right-hand side of \eqref{Eq: npp}, we obtain
\begin{align*}
\begin{pmatrix}
\alpha_0\\
\beta_0\\
1
\end{pmatrix}&=\frac{1}{\alpha_{k+1}\ldots\alpha_{k+h}}M_0M_1\ldots M_{k+h-1}(M_0\ldots M_{k-1})^{-1}\begin{pmatrix}
\alpha_{0}\\
\beta_{0}\\
1
\end{pmatrix}=\\
&=\frac{1}{\lambda'}\begin{pmatrix}  A_{k+h-1} & A_{k+h-2} & A_{k+h-3}\cr  B_{k+h-1} & B_{k+h-2} & B_{k+h-3} \cr C_{k+h-1} & C_{k+h-2} & C_{k+h-3}\end{pmatrix}\begin{pmatrix}  A_{k-1} & A_{k-2} & A_{k-3}\cr  B_{k-1} & B_{k-2} & B_{k-3} \cr C_{k-1} & C_{k-2} & C_{k-3}\end{pmatrix}^{-1}\begin{pmatrix}
\alpha_0\\
\beta_0\\
1
\end{pmatrix},
\end{align*}
where $\lambda'=\alpha_{k+1}\ldots\alpha_{k+h}$. Now, notice that
\begin{equation}
\begin{pmatrix}  A_{k-1} & A_{k-2} & A_{k-3}\cr  B_{k-1} & B_{k-2} & B_{k-3} \cr C_{k-1} & C_{k-2} & C_{k-3}\end{pmatrix}^{-1}=
\begin{pmatrix}  \tilde{B}_{k-2} & -\tilde{A}_{k-2} & \tilde{U}_{k-2}\cr  -\tilde{\tilde{B}}_{k-1} & \tilde{\tilde{A}}_{k-1} & -\tilde{\tilde{U}}_{k-1} \cr \tilde{B}_{k-1} & -\tilde{A}_{k-1} & \tilde{U}_{k-1}\end{pmatrix},
\end{equation}
where the latter sequences are as defined in Section \ref{Sec: jacobi}. Thus, we obtain
\begin{align*}
X_{1,1}\alpha_0+X_{1,2}\beta_0+X_{1,3}&=\lambda'\alpha_0,\\
X_{2,1}\alpha_0+X_{2,2}\beta_0+X_{2,3}&=\lambda'\beta_0,\\
X_{3,1}\alpha_0+X_{3,2}\beta_0+X_{3,3}&=\lambda',
\end{align*}
where
\begin{align*}
X_{1,1}&=A_{k+h-1}\tilde{B}_{k-2}-A_{k+h-2}\tilde{\tilde{B}}_{k-1}+A_{k+h-3}\tilde{B}_{k-1},\\
X_{1,2}&=-A_{k+h-1}\tilde{A}_{k-2}+A_{k+h-2}\tilde{\tilde{A}}_{k-1}-A_{k+h-3}\tilde{A}_{k-1},\\
X_{1,3}&=A_{k+h-1}\tilde{U}_{k-2}-A_{k+h-2}\tilde{\tilde{U}}_{k-1}+A_{k+h-3}\tilde{U}_{k-1},\\
X_{2,1}&=B_{k+h-1}\tilde{B}_{k-2}-B_{k+h-2}\tilde{\tilde{B}}_{k-1}+B_{k+h-3}\tilde{B}_{k-1},\\
X_{2,2}&=-B_{k+h-1}\tilde{A}_{k-2}+B_{k+h-2}\tilde{\tilde{A}}_{k-1}-B_{k+h-3}\tilde{A}_{k-1},\\
X_{2,3}&=B_{k+h-1}\tilde{U}_{k-2}-B_{k+h-2}\tilde{\tilde{U}}_{k-1}+B_{k+h-3}\tilde{U}_{k-1},\\
X_{3,1}&=C_{k+h-1}\tilde{B}_{k-2}-C_{k+h-2}\tilde{\tilde{B}}_{k-1}+C_{k+h-3}\tilde{B}_{k-1},\\
X_{3,2}&=-C_{k+h-1}\tilde{A}_{k-2}+C_{k+h-2}\tilde{\tilde{A}}_{k-1}-C_{k+h-3}\tilde{A}_{k-1},\\
X_{3,3}&=C_{k+h-1}\tilde{U}_{k-2}-C_{k+h-2}\tilde{\tilde{U}}_{k-1}+C_{k+h-3}\tilde{U}_{k-1}.\\
\end{align*}
We get rid of $\lambda'$, to obtain
\begin{align}
X_{1,1}\alpha_0+X_{1,2}\beta_0+X_{1,3}&=X_{3,1}\alpha_0^2+X_{3,2}\alpha_0\beta_0+X_{3,3}\alpha_0,\label{Eq: firstX} \\
X_{2,1}\alpha_0+X_{2,2}\beta_0+X_{2,3}&=X_{3,1}\alpha_0\beta_0+X_{3,2}\beta_0^2+X_{3,3}\beta_0.\label{Eq: secondX}
\end{align}
Since $\beta_0$ is not rational, then $X_{2,1}-X_{3,1}\beta_0\neq 0$. Hence from \eqref{Eq: secondX}, we obtain
\[\alpha_0=\frac{X_{3,2}\beta_0^2+(X_{3,3}-X_{2,2})\beta_0-X_{2,3}}{X_{2,1}-X_{3,1}\beta_0}.\]
We substitute the latter expression in \eqref{Eq: firstX} and we get
\begin{align*}
&X_{3,1}(X_{3,2}\beta_0^2+(X_{3,3}-X_{2,2})\beta_0-X_{2,3})^2-(X_{1,2}\beta_0+X_{1,3})(X_{2,1}-X_{3,1}\beta_0)^2+\\
&+(X_{3,2}\beta_0+X_{3,3}-X_{1,1})(X_{3,2}\beta_0^2+(X_{3,3}-X_{2,2})\beta_0-X_{2,3})(X_{2,1}-X_{3,1}\beta_0)=0.
\end{align*}
The two terms of degree $4$ are $X_{3,1}X_{3,2}^2\beta_0^4$ and $-X_{3,1}X_{3,2}^2\beta_0^4$, therefore the left-hand side is a polynomial of degree $3$, that is known to be irreducible \cite{Bren}. With some computations, it is possible to see that the equation reduces to
\[A\beta_0^3+B\beta_0^2+C\beta_0+D=0,\]
where:
\begin{align*}
A=&-X_{1,2}X_{3,1}^2+X_{1,1}X_{3,1}X_{3,2}-X_{2,2}X_{3,1}X_{3,2}+X_{2,1}X_{3,2}^2,\\
B=& \; 2X_{1,2}X_{2,1}X_{3,1}-X_{1,1}X_{2,2}X_{3,1}+X_{2,2}X_{3,1}^2-X_{1,3}X_{3,1}^2-X_{1,1}X_{2,1}X_{3,2}+\\
&-X_{2,1}X_{2,2}X_{3,2} -X_{2,3}X_{3,1}X_{3,2}+X_{1,1}X_{3,1}X_{3,3}-X_{2,2}X_{3,1}X_{3,3}+2X_{2,1}X_{3,2}X_{3,3},\\
C=&-X_{1,2}X_{2,1}^2+X_{1,1}X_{2,1}X_{2,2}+2X_{1,3}X_{2,1}X_{3,1}-X_{1,1}X_{2,3}X_{3,1}+2X_{2,2}X_{2,3}X_{3,1}+\\
&-X_{2,1}X_{2,3}X_{3,2}-X_{1,1}X_{2,1}X_{3,3}-X_{2,1}X_{2,2}X_{3,3}-X_{2,3}X_{3,1}X_{3,3}+X_{2,1}X_{3,3}^2 \\
D=&-X_{1,3}X_{2,1}^2+X_{1,1}X_{2,1}X_{2,3}+X_{2,3}^2X_{3,1}-X_{2,1}X_{2,3}X_{3,3}.
\end{align*}

Using Lemma \ref{Lem: numeden}, we can give an upper bound to the coefficients $X_{i,j}$ in terms of the denominator $C_{h+k-1}$. Notice that $h$ and $k$ are constant, because they represent the period and the pre-period of the continued fraction expansion of $(\alpha_0, \beta_0)$. It is possible to notice that
\begin{equation}\label{Eq: boundX}
X_{i,j}\leq 6 C_{h+k-1}^3.
\end{equation}
This gives a bound on the coefficients of the minimal polynomial of $\beta_0$, and then
\begin{equation}
H(\beta_0)\leq 2592 C_{h+k-1}^9.
\end{equation}
Notice that we used the fact that $A_n,B_n\leq C_n$, therefore we are handling the case when $0<\alpha_0 ,\beta_0 <1$.\bigskip

Now we want to apply the same reasoning to $\alpha_0$, in order to find its minimal polynomial and a bound on its height. From \eqref{Eq: firstX}, we obtain:
\[\beta_0=\frac{X_{3,1}\alpha_0^2+(X_{3,3}-X_{1,1})\alpha_0-X_{1,3}}{X_{1,2}-X_{3,2}\alpha_0}.\]
Notice again that the denominator is nonzero as $\alpha_0$ is irrational.

We substitute the latter expression in \eqref{Eq: secondX} and we get
\begin{align*}
&X_{3,2}(X_{3,1}\alpha_0^2+(X_{3,3}-X_{1,1})\alpha_0-X_{1,3})^2-(X_{2,1}\alpha_0+X_{2,3})(X_{1,2}-X_{3,2}\alpha_0)^2+\\
&+(X_{3,1}\alpha_0+X_{3,3}-X_{2,2})(X_{3,1}\alpha_0^2+(X_{3,3}-X_{1,1})\alpha_0-X_{1,3})(X_{1,2}-X_{3,2}\alpha_0)=0.\\
\end{align*}
The two terms of degree $4$ are $X_{3,2}X_{3,1}^2\alpha_0^4$ and $-X_{3,2}X_{3,1}^2\alpha_0^4$, therefore the left-hand side is a polynomial of degree $3$. With some computations, it is possible to see that the equation reduces to
\[A\alpha_0^3+B\alpha_0^2+C\alpha_0+D=0,\]
where
\begin{align*}
A=& \; X_{1,2}X_{3,1}^2-X_{1,1}X_{3,1}X_{3,2}+X_{2,2}X_{3,1}X_{3,2}-X_{2,1}X_{3,2}^2,\\
B=&-X_{1,1}X_{1,2}X_{3,1}-X_{1,2}X_{2,2}X_{3,1}+X_{1,1}^2X_{3,2}+2X_{1,2}X_{2,1}X_{3,2}-X_{1,1}X_{2,2}X_{3,2}+\\ 
&-X_{1,3}X_{3,1}X_{3,2}-X_{2,3}X_{3,2}^2+2X_{1,2}X_{3,1}X_{3,3} -X_{1,1}X_{3,2}X_{3,3} + X_{2,2}X_{3,2}X_{3,3},\\
C=& -X_{1,2}^2X_{2,1}+X_{1,1}X_{1,2}X_{2,2}-X_{1,2}X_{1,3}X_{3,1}+2X_{1,1}X_{1,3}X_{3,2} -X_{1,3}X_{2,2}X_{3,2} +\\
&+2X_{1,2}X_{2,3}X_{3,2} - X_{1,1}X_{1,2}X_{3,3} - X_{1,2}X_{2,2}X_{3,3}-X_{1,3}X_{3,2}X_{3,3}+X_{1,2}X_{3,3}^2,\\
D=&\; X_{1,2}X_{1,3}X_{2,2} -X_{1,2}^2X_{2,3}+X_{1,3}^2 X_{3,2}-X_{1,2}X_{1,3}X_{3,3}.
\end{align*}

Also in this case we use \eqref{Eq: boundX}, and we get that
\begin{equation}
H(\alpha_0)\leq 3024 C_{h+k-1}^9,
\end{equation}
which yields the thesis.
\end{proof}

\begin{Theorem}\label{Thm: heightGen}
Given $(\alpha_0, \beta_0) \in \mathbb R^2$ such that
\begin{equation}\label{Eq: pureper2}
(\alpha_0,\beta_0)=[(a_0,\ldots,a_{k-1},\overline{a_k,\ldots,a_{k+h-1}}),(b_0,\ldots,b_{k-1},\overline{b_k,\ldots,b_{k+h-1}})],
\end{equation}
then $\alpha_0$ and $\beta_0$ are cubic irrationals. Moreover, if $0 < \alpha_0 < N$, $0 < \beta_0 < M$, for some positive integers $N,M$, then
\[H(\alpha_0), H(\beta_0)\leq 3024 N^5M^5 C_{h+k-1}^9.\]
\end{Theorem}

\section{Multidimensional continued fractions with rapidly increasing partial quotients}\label{Sec: liouville}

In this section we provide some infinite families of multidimensional continued fractions converging to transcendental numbers, using a method similar to Liouville. First we recall Roth's theorem.

\begin{Theorem}[Roth]\label{Thm: Roth}
Let $\alpha\in\mathbb{R}$ be an algebraic number. Then, for all $\epsilon>0$, there exist only finitely many $(p,q)\in\mathbb{Z}$ such that
\[\left|\frac{p}{q}-\alpha\right|<\frac{1}{q^{2+\epsilon}}.\]
\end{Theorem}

\begin{Theorem}\label{Thm: LioTransJ}
Let $[(a_0, a_1, \ldots), (b_0, b_1, \ldots)]$ be two infinite admissible sequences of integer partial quotients for Jacobi's algorithm, hence converging to a pair of real numbers $(\alpha_0,\beta_0)$, and let $\delta>0$. If
\begin{equation}\label{Eq: unbound-jacobi}
a_{n}>\max\{|\tilde{A}_{n}|C_{n-1}^{\delta},|\tilde{B}_n|C_{n-1}^{\delta}\},
\end{equation}
for all $n\in\mathbb{N}$, then $\alpha_0$ and $\beta_0$ are (rationally independent) transcendental numbers.
\begin{proof}
Let us prove that, when condition \eqref{Eq: unbound-jacobi} holds, then there are infinitely many $n\in\mathbb{N}$ such that
\[\left|\alpha_0-\frac{A_n}{C_n}\right|<\frac{1}{C_n^w},
\]
and infinitely many $m\in\mathbb{N}$ such that
\[
\left|\beta_0-\frac{B_m}{C_m}\right|<\frac{1}{C_m^{w'}},
\]
with $w,w'>2$. In this way, $\alpha_0,\beta_0$ do not satisfy Roth's theorem and they are both transcendental numbers. By Lemma \ref{lemma:approx}, we have that, for infinitely many $n\in\mathbb{N}$,
\[\left|\alpha_0-\frac{A_n}{C_n}\right|<\left|\frac{A_{n+1}}{C_{n+1}}-\frac{A_{n}}{C_n}\right|=\left|\frac{\tilde{A}_{n+1}}{C_{n+1}C_{n}}\right|<\frac{|\tilde{A}_{n+1}|}{a_{n+1}C_n^2}<\frac{1}{C_n^{2+\delta}},\]
where in the last inequality we have exploited condition \eqref{Eq: unbound-jacobi}.
It follows that, by Roth's theorem, $\alpha$ is transcendental. Similarly as before, we can prove that
\[\left|\beta_0-\frac{B_n}{C_n}\right|<\left|\frac{B_{n+1}}{C_{n+1}}-\frac{B_{n}}{C_n}\right|=\left|\frac{\tilde{B}_{n+1}}{C_{n+1}C_{n}}\right|<\frac{|\tilde{B}_{n+1}|}{a_{n+1}C_n^2}<\frac{1}{C_n^{2+\delta}},\]
holds for infinitely many $n\in\mathbb{N}$, so that also $\beta_0$ is transcendental. Rational independence follows readily from Remark \ref{Rem:interruptions}.
\end{proof}
\end{Theorem}
\begin{Remark}

Condition \eqref{Eq: unbound-jacobi} of Theorem \ref{Thm: LioTransJ} can be practically used in order to construct pairs of transcendental numbers by means of Jacobi's algorithm. In fact, the sequences $\tilde{A}_n$ and $\tilde{B}_n$ do not depend on $a_n$, but only on $b_n$ and on previous partial quotients. Therefore, for any arbitrary choice of the sequence $b_n$, we have a different pair of transcendental numbers $(\alpha_0,\beta_0)$, that are rationally independent.
\end{Remark}

In the following, we exploit a similar technique in order to provide an $m$-tuple of transcendental numbers by means of Jacobi--Perron's algorithm for continued fractions in higher dimensions.

\begin{Theorem}\label{Thm: LiotransJP}
Let $[(a_n^{(1)})_{n\geq 0},\ldots,(a_n^{(m)})_{n\geq 0}]$ be $m$ infinite sequences of admissible partial quotients for Jacobi--Perron's algorithm, hence converging to an $m$-tuple of real numbers $(\alpha_0^{(1)},\ldots,\alpha_0^{(m)})$, and let $\delta>0$. If
\begin{equation}\label{Eq: unbound-jacobiperron}
a^{(1)}_{n}>\max\{|\tilde{A}_{n}^{(1)}|C_{n-1}^{\delta},|\tilde{A}^{(2)}_{n}|C_{n-1}^{\delta},\ldots,|\tilde{A}^{(m)}_{n}|C_{n-1}^{\delta}\},
\end{equation}
for all $n\in\mathbb{N}$, then all the real numbers $(\alpha_0^{(1)},\ldots,\alpha_0^{(m)})$ are transcendental.
\begin{proof}
The proof is similar to that of Theorem \ref{Thm: LioTransJ} for the classical Jacobi's algorithm. By Lemma \ref{lemma:approx}, we know that, for infinitely many $n\in\mathbb{N}$,
\[\left|\alpha_0^{(i)}-\frac{A^{(i)}_n}{C_n}\right|<\left|\frac{A^{(i)}_{n+1}}{C_{n+1}}-\frac{A^{(i)}_{n}}{C_{n}}\right|=\left|\frac{\tilde{A}^{(i)}_{n+1}}{C_{n+1}C_{n}}\right|<\frac{|\tilde{A}^{(i)}_{n+1}|}{a^{(1)}_{n+1}C_n^2}<\frac{1}{C_n^{2+\delta}}.\]
It follows, by Roth's theorem, that $\alpha_0^{(i)}$ is transcendental for all $i=1,\ldots,m$.
\end{proof}
\end{Theorem}

Thanks to Remark \ref{Rem: nondep}, it is possible to define recursively the sequence of partial quotients $a_n^{(1)}$ such that they satisfy \eqref{Eq: unbound-jacobiperron}, and hence the MCF converges to an $m$-tuple of transcendental numbers. Since, for $m \geq 3$, the infiniteness of the sequences of partial quotients does not guarantee in general that these $m$ real transcendental numbers are rationally independent, it  would be interesting to explore under which conditions this happens.

\section{Quasi-periodic multidimensional continued fractions}
\label{Sec: quasiper}

In this section, we deal with transcendence criteria for families of quasi-periodic MCFs.
The main results of Diophantine approximation that we exploit are the following two theorems, which %The first is an improvement of Roth's theorem, where it is provided an exact estimation of the finite number of solution \cite{DR}. Then, Theorem \ref{Thm: bugdiff} and Theorem \ref{Thm: bugsame}
provide two different generalizations of Liouville's theorem. In Theorem \ref{Thm: bugsame} we are requiring more hypotheses than Theorem \ref{Thm: bugdiff}, i.e. that the approximations lie all inside the same number field, and it will be used for the construction in the proof of Theorem \ref{Thm: Main2}.

%\begin{Theorem}[Davenport, Roth]
%Let $\alpha$ be an algebraic number. Then, for any $\epsilon>0$,
%\begin{equation}\label{Eq: roth}
%\left|\alpha-\frac{p_n}{q_n}\right|<\frac{1}{q_n^{2+\epsilon}},
%\end{equation}
%has only finitely many solutions $(p,q)\in \mathbb{Z}$. In particular, if $0<\epsilon\leq\frac{2}{3}$, then there are at most $e^{c_1\epsilon^{-2}}$ solutions, where $c_1$ is a constant depending only on $\alpha$.
%\end{Theorem}

\begin{Theorem}[\cite{BUG}, Theorem 2.6]\label{Thm: bugdiff}
Let $\alpha$ be an algebraic number and $n\geq 1$ be an integer. Then, for all $\epsilon>0$, there exists a constant $c$ such that
\[|\alpha-\xi|>\frac{c}{H(\xi)^{n+1+\epsilon}},\]
for all $\xi\neq \alpha$ of degree at most $n$.
\end{Theorem}

\begin{Theorem}[\cite{BUG}, Theorem 2.7]\label{Thm: bugsame}
Let $K$ be a number field and $\alpha\notin K$ be an algebraic number. Then, for all $\epsilon>0$, there exists a constant $c$ such that
\[|\alpha-\xi|>\frac{c}{H(\xi)^{2+\epsilon}},\]
for all $\xi\in K$.
\end{Theorem}

First of all, we prove a bound on the growth for the denominators of the convergents.
%of algebraic numbers, obtaining a result similar to a famous theorem of Davenport and Roth \cite{DR} for classical continued fractions.

\begin{Theorem}\label{Thm: davrothmulti}
    Let us consider $(\alpha_0^{(1)},\dots,\alpha_0^{(m)}) = [(a_0^{(1)},a_1^{(1)},\dots),\dots,(a_0^{(m)},a_1^{(m)},\dots)]$, and $d \geq 1$. If $a_{n+1}^{(1)}<C_n^d$ for all $n\in\mathbb{N}$, then there exists a constant $K=K(d,m)$ such that
\begin{equation}\label{Eq: conditionDR}
\log \log C_{n+1}<  Kn,
\end{equation}
for all $n \geq 0$. A suitable choice for $K$ is
\[  K = \log(d+1)+\log\left(1+\frac{1}{d}\right)+\log\log(m+1). \]
\begin{proof}
   Recall that, for all $n$,
    \begin{equation*}
       C_{n+1} = \sum_{j=1}^{m} a_{n+1}^{(j)}C_{n+1-j}+ C_{n-m}.
    \end{equation*}
    Since $(C_k)_{k\geq 0}$ is increasing and $a_{n+1}^{(j)} < C_n$ for all $j$ by hypothesis, together with \eqref{eq:cond-pq}, we get that
    \[C_{n+1} < (m+1) C_n^{d +1}.\]

    Now, by recursion, it is easily seen that
    \begin{multline*}
C_{n+1} < (m+1) C_n^{d + 1} < (m+1)^{1+(d+1)} C_{n-1}^{(d+1)^2} < \cdots < \\ 
< (m+1) ^{1 + (d+1) + \,\cdots\, + (d+1)^{n}} C_0^{(d+1)^{n+1}}  = (m+1)^{\frac{(d+1)^{n+1}-1}{d}},
    \end{multline*}
    which yields
    \[
        \log \log C_{n+1} <n\left( \log(d+1)+\log\left(1+\frac{1}{d}\right)+\log\log(m+1)\right),\]
    for all $n$, and the thesis follows.
\end{proof}
\end{Theorem}

\begin{Definition}[Quasi-periodic continued fractions]\label{Def: quasiperiodic}
Let us consider the sequences of positive integers $(\lambda_k)_{k\geq 0}$, $(r_k)_{k\geq 0}$ and $(n_k)_{k\geq 0}$, with $(n_k)_{k\geq 0}$ strictly increasing. Let us suppose that $(a_n^{(1)})_{n\geq 0},\ldots,(a_n^{(m)})_{n\geq 0}$ satisfy, for all $j=1,\ldots,m$,
\begin{equation}\label{Eq: quasiperiodic}
a_{m+r_k}^{(j)}=a_m^{(j)}, \ \ \text{        for } n_k\leq m\leq n_k+(\lambda_k-1)r_k-1
\end{equation}
The multidimensional continued fraction $[(a_n^{(1)})_{n\geq 0},\ldots,(a_n^{(m)})_{n\geq 0}]$, where the latter sequences satisfy \eqref{Eq: quasiperiodic} is called \textit{quasi-periodic}.
\end{Definition}

\begin{Remark}
The meaning of Definition \ref{Def: quasiperiodic} is well explained in the papers of Baker \cite{BAK} and Adamczewski-Bugeaud \cite{AB3} for the unidimensional case. Basically, for any $k\geq0$, at the index $n_k$ there starts a block of $r_k$ partial quotients, and this block is repeated exactly $\lambda_k$ times. The reason of the name \textit{quasi-periodic} is that, if $\lambda_k$ grows with $k$, then the blocks keep repeating for an increasing number of times, simulating the behavior of a periodic continued fraction. Notice also that, in Definition \ref{Def: quasiperiodic} for the multidimensional case, the sequences $(\lambda_k)_{k\geq 0}$, $(r_k)_{k\geq 0}$ and $(n_k)_{k\geq 0}$, with $(n_k)_{k\geq 0}$ are the same for all the $(a_n^{(j)})_{n\geq 0}$, $j=1,\ldots,m$.
\end{Remark}

\begin{Theorem}\label{Thm: Main1}
Let $(\alpha,\beta)=[(a_0,a_1,\ldots),(b_0,b_1,\ldots)]$ be a quasi-periodic continued fraction such that, for some $d\geq 1$, $a_{i+1}<C_i^d$ for all $i\in\mathbb{N}$. Moreover, let us suppose that $r_k<cn_k$. where $c$ is a constant. If
\begin{equation}\label{Eq: hpbaker}
\lim\limits_{k\rightarrow \infty}\frac{\log \lambda_k}{n_k}=+\infty,
\end{equation}
then $\alpha$ and $\beta$ are transcendental numbers.
\end{Theorem}
\begin{proof}
Let us suppose by contradiction that $\alpha$ and $\beta$ are algebraic numbers. Starting from the \textit{quasi-periodic} expansion of $(\alpha,\beta)$, we construct a sequence of pairs $(\eta^{(k)},\zeta^{(k)})_{k\geq 0}$ in the following way: 
\begin{align}
\eta^{(k)}&=[a_0,\ldots a_{n_k-1},\overline{a_{n_k},\ldots,a_{n_k+r_k-1}}],\label{Eq: etak}\\
\zeta^{(k)}&=[b_0,\ldots b_{n_k-1},\overline{b_{n_k},\ldots,b_{n_k+r_k-1}}].\label{Eq: zetak}
\end{align}
for all $k\geq 0$.
For simplicity let us set, for all $k\geq 1$,
\[m_k=n_k+(\lambda_k-1)r_k-1.\]
Then, the elements $\eta^{(k)}$ and $\zeta^{(k)}$ are cubic irrationalities for all $k$, that have in common with $\alpha$ and $\beta$, respectively, the first $m_k$ partial quotients. Therefore, by Lemma \ref{Lem: equalpq}, Theorem \ref{Thm: height} and  Theorem \ref{Thm: bugdiff}, choosing some $\epsilon>0$,
\begin{align*}
\frac{2}{C_{m_k}}>|\alpha-\eta^{(k)}|>\frac{c_1}{C_{n_k+r_k-1}^{9(4+\epsilon)}},\\
\frac{2}{C_{m_k}}>|\beta-\zeta^{(k)}|>\frac{c_2}{C_{n_k+r_k-1}^{9(4+\epsilon)}}.\\
\end{align*}
From the latter equations and Lemma \ref{Lem: psi}, we obtain 
\begin{equation}
\psi^{m_k-2}<C_{m_k}<c_3C_{n_k+r_k-1}^{9(4+\epsilon')},
\end{equation}
that is
\begin{equation}\label{Eq: nkrk}
\lambda_k<m_k-2<c_4+c_5\log (C_{n_k+r_k-1})<c_6\log (C_{n_k+r_k-1}).
\end{equation}
Therefore, using only Theorem \ref{Thm: bugdiff} we have obtained that, if either $\alpha$ or $\beta$ is an algebraic number, then
\begin{equation}\label{Thm: contr1}
\log\log(C_{n_k+r_k-1})>c_7\log \lambda_k.
\end{equation}
However, by Theorem \ref{Thm: davrothmulti},
\begin{equation}\label{Thm: contr2}
\log \log (C_{n_k+r_k-1})<c_8(n_k+r_k-1).
\end{equation}
Plugging together \eqref{Thm: contr1} and \eqref{Thm: contr2}, and since $r_k<cn_k$ by hypothesis, we obtain
\[\frac{\log \lambda_k}{n_k}<c_9,\]
that is a contradiction with \eqref{Eq: hpbaker}. Therefore, $\alpha$ and $\beta$ are transcendental numbers.
\end{proof}

\begin{Theorem}\label{Thm: Main2}
Let $(\alpha,\beta)=[(a_0,a_1,\ldots),(b_0,b_1,\ldots)]$ be a \textit{quasi-periodic} continued fraction. Let us suppose that there exist two positive integers $M,N$ such that $a_{k},b_k\leq M$ and $r_k\leq N$ for all $k\in\mathbb{N}$. Let us define the constant
\[B=2\frac{\log\eta}{\log \psi}-1,\]
where $\eta$ and $\psi$ are taken as in Lemma \ref{Lem: equalpq}, i.e. the positive real roots of $x^3-Mx^2-Mx-1$ and $x^3-x^2-x-1$, respectively.
Then, if
\begin{equation}\label{Eq: Hpbak2}
\limsup\limits_{k\rightarrow \infty}\frac{\lambda_k}{n_k}>B,
\end{equation}
both $\alpha$ and $\beta$ are transcendental numbers.
\end{Theorem}

\begin{proof}
Let us suppose by contradiction that either $\alpha$ or $\beta$ is algebraic. Since both the length of the periods and the partial quotients are bounded by hypothesis, i.e. $a_n,b_n\leq M$ and  $r_n\leq N$, there exist only finitely many possible different blocks, as there are only finitely many pairs $(a_{n_k},b_{n_k}),\ldots,(a_{n_k+r_k-1},b_{n_k+r_k-1})$. Since in the \textit{quasi-periodic} continued fraction \eqref{Eq: quasiperiodic} there are infinitely many pairs of blocks that repeat, we can find an infinite subsequence $(k_l)_{l\geq 0}$, such that the blocks starting at the partial quotients $a_{n_{k_l}}$ and $b_{n_{k_l}}$ are the same for all $l\geq 0$. Therefore, we can find a sequence of algebraic numbers
\begin{align*}
\xi^{(t)}&=[a_0,\ldots, a_{n_t-1},\overline{a_{n_t},\ldots,a_{n_t+r_t-1}}],\\
\gamma^{(t)}&=[b_0,\ldots, b_{n_t-1},\overline{b_{n_t},\ldots,b_{n_t+r_t-1}}],
\end{align*}

where $t=k_l$  for all $l\geq 0$. This is the multidimensional version of the construction in the proof of Theorem 2 of \cite{BAK}. The peculiarity of these cubic irrationalities, that form a  subsequence of the $\eta^{(k)}$ and $\zeta^{(k)}$ defined in \eqref{Eq: etak} and \eqref{Eq: zetak}, is that they share the same periodic part for all $t$. Therefore, by Lemma \ref{Lem: samecubicfield}, they lie in the same cubic field. In this case, we can use Theorem \ref{Thm: height} and Theorem \ref{Thm: bugsame} together with Lemma \ref{Lem: equalpq}. Choosing $\epsilon>0$, and calling $m_t=n_t+(\lambda_t-1)r-1$, we have
\begin{align*}\label{Eq: alphak}
\frac{2}{C_{m_t}}&>|\alpha-\xi^{(t)}|>\frac{c_1}{C_{n_t+r-1}^{9(2+\epsilon)}},\\
\frac{2}{C_{m_t}}&>|\beta-\gamma^{(t)}|>\frac{c_2}{C_{n_t+r-1}^{9(2+\epsilon)}}.
\end{align*}
This means that
\begin{equation}
C_{m_t}<c_3C_{n_t+r-1}^{9(2+\epsilon)}.
\end{equation}
From Lemma \ref{Lem: psi}, we obtain
\begin{equation}
\psi^{m_t-2}<\eta^{9(2+\epsilon)(n_t+r-1)},
\end{equation}
that is
\begin{equation}\label{Eq: B}
n_t+(\lambda_t-1)r-3<9(2+\epsilon)(n_t+r-1)\frac{\log \eta}{\log \psi}.
\end{equation}
Since $r$ and $\epsilon$ are constant, we can rewrite \eqref{Eq: B} as
\begin{equation}
n_t+r\lambda_t<c_4+9(2+\epsilon)n_t\frac{\log\eta}{\log\psi},
\end{equation}
from which we obtain
\begin{equation}
\lambda_t<r\lambda_t<c_4+n_t\left(9(2+\epsilon)\frac{\log\eta}{\log\psi}-1\right).
\end{equation}
We get,
\[\frac{\lambda_t}{n_t}<\frac{c_4}{n_t}+\left(18\frac{\log\eta}{\log\psi}-1\right)+9\epsilon\frac{\log\eta}{\log\psi},\]

and the latter holds for all $\epsilon>0$. Therefore for sufficiently small $\epsilon$ and sufficiently large $t$, it contradicts the hypothesis \eqref{Eq: Hpbak2}.
\end{proof}

\section{Open problems and further research}
\label{Sec: openprob}
In this section, we collect some open problems and hints for further research, that can allow to obtain more general or stronger results.\bigskip

In the previous section we proved some transcendence criteria for quasi-periodic MCFs when $m = 2$. An obstruction that we found in proving Theorem \ref{Thm: Main1} and \ref{Thm: Main2} for MCFs of higher dimension is an analogue of Theorem \ref{Thm: height}. We leave it here as a conjecture.

\begin{Conjecture}
Given $(\alpha_0^{(1)},\ldots, \alpha_0^{(m)}) \in \mathbb R^m$ having periodic continued fraction
\begin{equation}\label{Eq: pureperN}
\left[\left(a_0^{(1)},\ldots,a_{k-1}^{(1)},\overline{a_k^{(1)},\ldots,a_{k+h-1}^{(1)}}\right),\ldots,\left(a_0^{(m)},\ldots,a_{k-1}^{(m)},\overline{a_k^{(m)},\ldots,a_{k+h-1}^{(m)}}\right)\right],
\end{equation}
then $\alpha_0^{(1)},\ldots, \alpha_0^{(m)}$ are algebraic irrationalities of degree at most $m+1$. Moreover, if $0 < \alpha_0^{(i)}$ for all $i=1,\ldots,m$, then there exist two positive constants $L_1,L_2\in\Z$ which do not depend on any $\alpha_0^{(i)}$ and such that
\[H(\alpha_0^{(i)})\leq L_1C_{h+k-1}^{L_2},\]
for all $i=1,\ldots,m$.
\end{Conjecture}
In order to prove the conjecture, similarly as in \eqref{Eq: pp}, it is not difficult to see that
\begin{equation}\label{Eq: pp2}
\begin{pmatrix}
\alpha_0^{(1)}\\
\alpha_0^{(2)}\\
\vdots\\
\alpha_0^{(m)}\\
1
\end{pmatrix}=\frac{1}{\alpha_1^{(1)}\alpha_2^{(1)}\ldots\alpha_{n+1}^{(1)}}\prod\limits_{j=0}^n\begin{pmatrix} a_j^{(1)} & 1 & 0 &\ldots & 0 \cr a_j^{(2)} & 0 & 1 &\ldots & 0  \cr \vdots & \vdots & \vdots &\ldots & \vdots \cr a_j^{(m)} & 0 & 0 &\ldots & 1  \end{pmatrix}\begin{pmatrix}
\alpha_{n+1}^{(1)}\\
\alpha_{n+1}^{(2)}\\
\vdots\\
\alpha_{n+1}^{(m)}\\
1
\end{pmatrix}.
\end{equation}
Following the proof of Theorem \ref{Thm: height}, we arrive to a system of equations
\begin{align*}
\begin{cases}
X_{1,1}\alpha_0^{(1)}+X_{1,2}\alpha_0^{(2)}+\ldots+X_{1,m}\alpha_0^{(m)}+X_{1,m+1}=\lambda'\alpha_0^{(1)},\\
X_{2,1}\alpha_0^{(1)}+X_{2,2}\alpha_0^{(2)}+\ldots+X_{2,m}\alpha_0^{(m)}+X_{1,m+1}=\lambda'\alpha_0^{(2)},\\
\vdots\\
X_{m,1}\alpha_0^{(1)}+X_{1,2}\alpha_0^{(2)}+\ldots+X_{1,m}\alpha_0^{(m)}+X_{1,m+1}=\lambda'\alpha_0^{(m)},\\
X_{m+1,1}\alpha_0^{(1)}+X_{m+1,2}\alpha_0^{(2)}+\ldots+X_{m+1,m}\alpha_0^{(m)}+X_{m+1,m+1}=\lambda',
\end{cases}
\end{align*}
where $\lambda'=\alpha_{k+1}^{(1)}\ldots\alpha_{k+h}^{(1)}$ and the $X_{i,j}$ are coefficients in $A_{n}^{(j)}$ and $C_n$, for $n=k+h-(m+1),\ldots,k+h-1$ and $j=1,\ldots,m$. Eliminating $\lambda'$ from the last equation, we can rewrite the system as

    \begin{align*}
    \begin{cases}
    \sum\limits_{i=1}^m X_{1,i}\alpha_0^{(i)}+X_{1,m+1}=\alpha_0^{(1)}\left(\sum\limits_{i=1}^m X_{m+1,i}\alpha_0^{(i)}+X_{m+1,m+1}\right)\\
    \sum\limits_{i=1}^m X_{2,i}\alpha_0^{(i)}+X_{2,m+1}=\alpha_0^{(2)}\left(\sum\limits_{i=1}^m X_{m+1,i}\alpha_0^{(i)}+X_{m+1,m+1}\right)\\
    \vdots\\
    \sum\limits_{i=1}^m X_{m,i}\alpha_0^{(i)}+X_{m,m+1}=\alpha_0^{(m)}\left(\sum\limits_{i=1}^m X_{m+1,i}\alpha_0^{(i)}+X_{m+1,m+1}\right).
    \end{cases}
    \end{align*}
The problem here is that, in this system of inequalities, it is not easy to write explicitly the minimal polynomials of $\alpha_0^{(i)}$ in order to bound its naive height. However, by using some elimination techniques on the previous system, it is possible to prove that there exist polynomials $P_i \in \mathbb{Q}[Y_{1,1},\dots,Y_{m+1,m+1},Z]$, for $i= 1,\ldots,m$, whose coefficients depend only on $m$, such that $P_i(X_{1,1},\dots,X_{m+1,m+1}, \alpha_0^{(i)})=0$. Hence the thesis would readily follow from a generalization of Lemma \ref{Lem: numeden}, which represents the only concrete obstruction.

Moreover, another issue in generalizing the results for quasi-periodic MCFs is that an analogue of Lemma \ref{Lem: Recursion} does not generally hold for MCFs with $m \geq 3$. In particular, for $m\geq 3$, it is not true that $|\tilde{A}^{(i)}_n|\leq C_n$ for all $n\in\N$, where the sequences $\tilde{A}^{(i)}_n$ are as in Section \ref{Sec: liouville}.

%Another aspect that is interesting to deepen in future works is the quality of the estimation obtained in Theorem \ref{Thm: davrothmulti} for the denominators of MCFs of algebraic numbers. In fact, in the classical result of Davenport and Roth, Condition \eqref{Eq: conditionDR} has the term $\sqrt{\log n}$ at the denominator of the right-hand side. Obtaining the same estimation for MCFs would be a stronger result than Theorem \ref{Thm: davrothmulti}, also allowing to relax the conditions for transcendence in Theorems \ref{Thm: Main1} and \ref{Thm: Main2}. The main obstruction for obtaining a stronger result is due to the worse approximation properties of the convergents of MCFs compared with classical continued fractions. In fact, in order to have $\frac{n}{\sqrt{\log n}}$ as main term in \eqref{Eq: DavRothstima}, the term $n\log 3$ should cancel out. This problem comes from \eqref{Eq: C1}, where the presence of $|\tilde{A}_{n+1}|$ allows to obtain, using Lemma \ref{Lem: tilde}, only $C_{n+1}\leq C_{n}^{3+\epsilon}$, instead of $C_{n+1}\leq C_{n}^{1+\epsilon}$ of the unidimensional case.\bigskip

In the main transcendence results of the paper, i.e. Theorems \ref{Thm: LioTransJ} and \ref{Thm: LiotransJP} for Liouville-type MCFs and Theorems \ref{Thm: Main1} and \ref{Thm: Main2} for quasi-periodic MCFs, we provide several criteria to produce a MCF converging to a $m$-tuple of rational independent transcendental numbers. We expect that in general, using our approach, the $m$ transcendental numbers are algebraically independent over $\mathbb{Q}$. It would be nice to prove such conjecture or to provide some sufficient conditions ensuring it.

After Liouville-type continued fractions \cite{LIO} and quasi-periodic continued fractions \cite{BAK,MAI}, other transcendence criteria have been proved for unidimensional continued fractions. Adamczewski and Bugeaud proved the transcendence of \textit{stammering} \cite{AB2} and \textit{palindromic} \cite{AB3} continued fractions. Later, the transcendence of several other families of continued fractions has been established by Adamczewski, Bugeaud and Davison \cite{ABD}. This analysis culminates in the proof of the transcendence for any automatic continued fraction, provided by Bugeaud \cite{BUG2}. In future works, it would be interesting to investigate and adapt these results in the multidimensional setting, possibly defining also some multidimensional generalization of the famous sequences considered in \cite{ABD}.

%Vale $|\tilde{A}_n|\leq C_n$ (Schweiger) ma non per frazioni continue di dimensione superiore\bigskip

%Spiegare come generalizzare la dimostrazione delle quasi periodiche una volta dimostrati questi due irsultati sopra.\bigskip

%Commenti su stammering e palindromiche (non vale il teorema di palindormicità che vale nel caso unidimensionale), da analizzare in lavori futuri.\bigskip

%Per continuare ulteriormente oltre stammering e palindromiche: Davison, Automatic, Rudin-Shapiro, Baum Sweet, Folded, Generalized perturbed symmmetry systems. Non fanno la Thue-Morse ma noi potremmo includerla. Una cosa bella potrebbe essere capire se di queste sequenze esista un analogo bidimensionale. In caso possiamo definirlo noi e dimostrare la trascendenza della frazione continua multidimensionale associata.\bigskip

\section*{Acknowledgments}
The authors are members of GNSAGA of INdAM.\\
The third author acknowledges that this study was carried out within the MICS (Made in Italy – Circular and Sustainable) Extended Partnership and received funding from the European Union Next-Generation EU (Piano Nazionale di Ripresa e Resilienza (PNRR) – Missione 4 Componente 2, Investimento 1.3 – D.D. 1551.11-10-2022, PE00000004).


\begin{thebibliography}{}

\bibitem{AB}{B. Adamczewski, Y. Bugeaud},
\textit{On the complexity of algebraic numbers I. Expansions in integer bases},
\newblock Ann. of Math. (2007), 547-565.

\bibitem{AB2}{B. Adamczewski, Y. Bugeaud},
\textit{On the complexity of algebraic numbers II. Continued fractions},
\newblock Acta Math., \textbf{195} (2005), 1-20.

\bibitem{AB3}{B. Adamczewski, Y. Bugeaud},
\textit{Palindromic continued fractions},
\newblock Ann. de l’Inst. Four. \textbf{57}, (2007), 1557-1574.


\bibitem{MB}{B. Adamczewski, Y. Bugeaud},
\textit{On the Maillet-Baker continued fractions},
\newblock Crelle Journal (2007), placeholder.

\bibitem{ABD}{B. Adamczewski, Y. Bugeaud, J. L. Davison},
\textit{Continued fractions and transcendental numbers},
\newblock Ann. de l’Inst. Four. \textbf{56}, (2006), 2093-2113.


\bibitem{BAK}{S. Baker},
\textit{Continued fractions of transcendental numbers},
\newblock Mathematika, \textbf{9} (1962), 1–8.

\bibitem{BEH}{R. Belhadef, H. A. Esbelin, T. Zerzaihi},
\textit{Transcendence of Thue–Morse $p$-adic continued fractions},
\newblock Mediterr. J. Math. \textbf{13}(4) (2016), 1429-1434.


\bibitem{BER}{L. Bernstein},
\textit{The Jacobi-Perron Algorithm -- Its Theory and Application},
\newblock Lecture Notes in Mathematics  \textbf{207}, Springer-Verlag, Berlin-New York, (1971).

\bibitem{BS}{D. Badziahin, J. Shallit}
\textit{An unusual continued fractions},
\newblock Proceedings of the American Mathematical Society \textbf{144(5)} (2016), 1887-1896.

\bibitem{Bren}{A. J. Brentjes}
\textit{Multi-dimensional continued fraction algorithms},
\newblock Mathematisch Centrum, Amsterdam, (1981).


\bibitem{BUG}{Y. Bugeaud},
\textit{Approximation by algebraic numbers},
\newblock Cambridge University Press (2004). 


\bibitem{BUG2}{Y. Bugeaud},
\textit{Automatic continued fractions are transcendental or quadratic},
\newblock Annales scientifiques de l'École Normale Supérieure. \textbf{46}(6) (2013).


\bibitem{DR}{H. Davenport, K. Roth},
\textit{Rational approximations to algebraic numbers},
\newblock Mathematika \textbf{2}(2) (1955), 160-167.

\bibitem{DFL04} E. Dubois, A. Farhane, R. Paysant-Le Roux,
\textit{Etude des interruptions dans l'algorithme de Jacobi-Perron},
\newblock Bull. Austral. Math. Soc. \textbf{69} (2004), 241-254.

\bibitem{HER}{C. Hermite},
\textit{Extraits de lettres de M. Ch. Hermite à M. Jacobi sur différents objects de la théorie des nombres},
\newblock J. Reine Angew. Math. \textbf{40}, (1850), 261-278.

\bibitem{JAC}{C. G. J. Jacobi},
\textit{Ges. Werke VI},
\newblock Berlin Academy, (1891), 385-426.

\bibitem{OK1}{O. Karpenkov},
\textit{On Hermite's problem, Jacobi-Perron type algorithms, and Dirichlet groups},
\newblock  Acta Arithm. \textbf{203}(1) (2022), 27-48.

\bibitem{OK2}{O. Karpenkov},
\textit{On a periodic Jacobi-Perron type algorithm},
\newblock Monatshefte für Mathematik, \textbf{205}(3) (2024), 531-601.

%\bibitem{KHI}{A. Ya. Khinchin},
%\textit{Continued fractions}, 
%\newblock University of Chicago Press (1964).

\bibitem{LV}{V. Laohakosol, P. Ubolsri},
\textit{$p$-adic continued fractions of Liouville type},
\newblock Proc. Amer. Math. Soc.  \textbf{101}(3) (1985), 403-410.

%\bibitem{LeV}{W. J. LeVeque},
%\textit{Topics in Number Theory},
%\newblock volumes I and II, Courier Corporation (2012).

\bibitem{LIO}{J. Liouville},
\textit{Sur des classes très-étendues de quantités dont la valeur n'est ni algébrique ni même réductible à des irrationnelles algébriques},
\newblock C.R. Acad. Sci. Paris, \href{http://gallica.bnf.fr/ark:/12148/bpt6k2977n/f883.image}{883-885}, \href{http://gallica.bnf.fr/ark:/12148/bpt6k2977n/f910.image}{910-911}.

\bibitem{LMS}{I. Longhi, N. Murru, F. M. Saettone},
\textit{Heights and transcendence of $p$-adic continued fractions},
\newblock Annali di Matematica Pura ed Applicata \textbf{204} (20025), 129--145.



%\bibitem{MAH}{K. Mahler}, 
%\textit{On the approximation of $\pi$}, 
%\newblock Nederl. Akad. Wetensch. Proc. Ser. A. \textbf{56}(1),  1953.

\bibitem{MAI}{E. Maillet}, 
\textit{Introduction à la théorie des nombres transcendants et des propriétés arithmétiques des fonctions}, 
\newblock Gauthier-Villars, Paris (1906).


\bibitem{OO}{T. Ooto},
\textit{Transcendental $p$-adic continued fractions}, 
\newblock Math. Z. \textbf{287} (2017), no. 3-4, 1053-1064.

\bibitem{PER}{O. Perron},
\textit{Grundlagen für eine theorie des Jacobischen kettenbruchalgorithmus}, 
\newblock Math. Ann. \textbf{64}(1) (1907), 1-76.

\bibitem{QUE}{M. Queffélec},
\textit{Transcendance des fractions continues de Thue–Morse}, 
\newblock J. Number Theory \textbf{73}(2) (1998), 201-211.


\bibitem{R}{G. Romeo},
\textit{Continued fractions in the field of $p$-adic numbers},
\newblock Bull. Amer. Math. Soc. (N.S.) \textbf{61}(2) (2024), 343-371.

\bibitem{ROTH}{K. F. Roth},
\textit{Rational approximations to algebraic numbers},
\newblock Mathematika, \textbf{2}(1) (1955), 1-20.

%\bibitem{SCHL}{H. Schlickewei},
%\textit{On products of special linear forms with algebraic coefficients},
%\newblock Acta Arithm., \textbf{31}(4) (1976), 389-398.

\bibitem{SCHM}{W. M. Schmidt},
\textit{On simultaneous approximations of two algebraic numbers by rationals},
\newblock Acta Math., \textbf{119} (1967), 27-50.

\bibitem{SCHM2}{W. M. Schmidt},
\textit{Diophantine approximation},
\newblock Springer,  Vol. 785 (2009).


\bibitem{MUL}{F. Schweiger},
\textit{Multidimensional continued fractions},
\newblock Oxford Science Publications, Oxford University Press, Oxford, (2000).

\bibitem{MUL2}{F. Schweiger},
\textit{The Metrical Theory of Jacobi-Perron Algorithm},
\newblock Lecture Notes in Mathematics \textbf{334} (1973).

\bibitem{TAM1}{J. Tamura},
\textit{A class of transcendental numbers having explicit $g$-adic expansion and the Jacobi-Perron algorithm},
\newblock Acta Arith. \textbf{61}(1) (1992), 51-67.

\bibitem{TAM2}{J. Tamura},
\textit{A class of transcendental numbers having explicit $g$-adic and Jacobi-Perron expansions of arbitrary dimension},
\newblock Acta Arith. \textbf{71}(4) (1995), 301-329.


\end{thebibliography}
\end{document}